\documentclass[a4paper,12pt]{article}
\pdfoutput=1
\usepackage{epsf,exscale,times}
\usepackage[english]{babel}
\usepackage{graphicx}
\usepackage{amsmath,amsfonts,amssymb,amsthm}
\usepackage{dsfont}
\usepackage{geometry}
\usepackage{subfig}
\usepackage{algorithm}
\usepackage{algpseudocode}

\graphicspath{{images/}}
\numberwithin{equation}{section}

\def\eqlaw{\buildrel d \over =}

\def\ps{\buildrel a.s. \over \rightarrow}
\newcommand{\myvecz}[2] {\left(\begin{array}{cc}#1\\ #2\end{array}\right)}
\newcommand{\myvec}[2] {\left(\begin{array}{cc}#1\\ \vdots \\ #2\end{array}\right)}
\newcommand{\myvecs}[3] {\left(\begin{array}{cc}#1\\#2\\ \vdots \\ #3\end{array}\right)}
\newtheorem{remark}{Remark}
\newtheorem{definition}{Definition}
\newtheorem{theorem}{Theorem}
\newtheorem{prop}{Proposition}
\newtheorem{lemma}{Lemma}
\newtheorem{example}{Example}

\title{Multivariate quantiles and multivariate L-moments}
\date{ }

\begin{document} 

\maketitle
\begin{abstract}
Univariate L-moments are expressed as projections of the quantile function onto an orthogonal basis of polynomials in $L_2([0;1],\mathbb{R})$. We present multivariate versions of L-moments expressed as collections of orthogonal projections of a multivariate quantile function on a basis of multivariate polynomials in  $L_2([0;1]^d,\mathbb{R})$. We propose to consider quantile functions defined as transport from the uniform distribution on $[0;1]^d$ onto the distribution of interest. In particular, we present the quantiles defined by the transport of Rosenblatt and the optimal transport and the properties of the subsequent L-moments.
\end{abstract}

\tableofcontents

%%------------------------------------
%% The introduction starts here
%%------------------------------------

\section{Motivations and notations}

Univariate L-moments are either expressed as sums of order statistics or as projections of the quantile function onto an orthogonal basis of polynomials in $L_2([0;1],\mathbb{R})$. Both concepts of order statistics and of quantile are specific to dimension one which makes non immediate a generalization to multivariate data.\\
Let $r\in\mathbb{N}_*:=\mathbb{N}\setminus \{0\}$. For an identically distributed sample $X_1,..., X_r$ on $\mathbb{R}$, we note $X_{1:r}\leq ... \leq X_{r:r}$ its order statistics. It should be noted that $X_{1:r},...,X_{r:r}$ are still random variables.\\
Then, if $\mathbb{E}[|X|]<\infty$, the $r$-th L-moment is defined by :
\begin{equation}
\lambda_r = \frac{1}{r} \sum_{k=0}^{r-1} (-1)^k \dbinom{r-1}{k} \mathbb{E}[X_{r-k:r}].
\label{eq:L_mom_ch2}
\end{equation}
If we use $F$ to denote the cumulative distribution function (cdf) and define the quantile function for $t\in[0;1]$ as the generalized inverse of $F$ i.e. $Q(t) = \inf\{x\in\mathbb{R} \text{ s.t. } F(x) > t\}$, this definition can be written : 
\begin{equation}
\lambda_r  = \int_0^1 Q(t) L_{r}(t)dt
\label{eq:iintlmom}
\end{equation}
where the $L_r$'s are the shifted Legendre polynomials which are a Hilbert orthogonal basis for $L^2([0;1],\mathbb{R})$ equipped with the usual scalar product (for $f,g\in L^2([0;1],\mathbb{R})$, $\langle f,g\rangle = \int_0^1 f(t)g(t)dt$) :
\begin{equation}
L_r(t)= \sum_{k=0}^{r-1} (-1)^{k} \dbinom{r-1}{k}^2 t^{r-1-k}(1-t)^k= \sum_{k=0}^{r-1} (-1)^{r-k} \dbinom{r-1}{k} \dbinom{r-1+k}{k} t^k.
\label{legendre_pol}
\end{equation}

L-moments were introduced by Hosking \cite{hosking90} in 1990 as alternative descriptors to central moments for a univariate distribution. They have some properties that we wish to keep for the analysis of multivariate data. Serfling and Xiao \cite{xiao07} listed the following key features of univariate L-moments which are desirable for a multivariate generalization : 
\begin{itemize}
\item The existence of the $r$-th L-moment for all $r$ if the expectation of the underlying random variable is finite
\item A distribution is characterized by its infinite series of L-moments (if the expectation is finite)
\item A scalar product representation with mutually orthogonal weight functions (equation \ref{eq:iintlmom})
\item A representation as expected value of an L-statistic (linear function of order statistics)
\item The U-statistic structure of sample versions should give asymptotic results
\item The L-statistic structure of sample versions should give a quick computation
\item Tractable unbiased sample version coming from the U-statistic and L-statistic structure should exist
\item Sample L-moments are more stable than classical moments, increasingly with higher order : the impact of each outlier is linear in the L-moment case whereas it is in the order of $(x-\bar{x})^k$ for classical moments of $k$ order
\end{itemize}
We will add two more properties related to the previous list :
\begin{itemize}
\item the equivariance of the L-moments with respect to the dilatation and their invariance with respect to translation for L-moments of an order larger than two
\item the tractability of the L-moments in some parametric families which makes them useful for estimation in these families, especially for the shape parameter of heavy tailed distributions.
\end{itemize}

Heavy-tailed distributions naturally appear in many different fields which then need description features for dispersion or kurtosis usually assuming moments with order larger than two; for example in applications in climatology based on annual data such as annual maximum rainfall. In \cite{hosking97}, Hosking and Wallis successfully applied univariate L-moments for the inference in the so-called regional frequency analysis that have to deal with heavy-tailed distributions. We can mention furthermore financial risk analysis \cite{jurczenko06} or target detection in radar \cite{ward06} that are fields in which multivariate heavy-tailed distributions appear.\\

Serfling and Xiao proposed a multivariate extension of L-moments for a vector $(X_1,...,X_d)^T$, based on the conditional distribution of $X_i$ given $X_j$ for all  $(i,j)\in\{1,...,d\}^2$. Their definition satisfies most of the properties of the univariate L-moments, but for the characterization of the multivariate distributions by the family of its L-moments. We generalize their approach by a slightly shift in perspective that will allow us to maintain the characterization property in the multivariate case.\\
Our starting point for a definition of multivariate L-moments is the characterization as orthogonal projection of the quantile onto an orthogonal basis of polynomials defined on $[0;1]$. It is not difficult to define orthogonal multi-indice polynomials on $[0;1]^d$ (see Lemma \ref{legendre}). It subsequently remains to define a multivariate quantile.\\
As there is no total order in $\mathbb{R}^d$, there are many different ways to define a multivariate quantile. Serfling made a survey of the existing approaches \cite{serfling02}. Amongst them, we can cite Chaudhuri's spatial quantiles \cite{chaudhuri96}, Zuo and Serfling's depth-based quantiles \cite{zuo00} or the generalized quantile process of Einmahl and Mason \cite{einmahl92}. In the DOQR (for Depth-Outlyingness-Quantile-Rank) paradigm given by Serfling \cite{serfling10}, multivariate quantiles map the ball of center zero and radius 1 $B_d(0,1)$ into $\mathbb{R}^d$ without specifying the norm underlying the ball. The definition of an orthogonal basis of polynomials is natural only in $[0;1]^d$, so we consider only the shifted unit ball for the infinite norm in our proposition of multivariate quantile.\\
The approach of multivariate quantile that has been chosen uses the notion of transport of measure. Indeed, in the univariate case, the quantile maps the uniform measure on $[0;1]$ onto the distribution of interest. Galichon and Henry \cite{galichon12} for example proposed to keep this basic property in order to define a multivariate quantile as the optimal transport between the uniform measure on $[0;1]^d$ and the multivariate distribution. We will adopt this definition by relaxing the optimality of the transport. Furthermore, if we consider the Rosenblatt transport \cite{rosenblatt52} in our definition of multivariate L-moments for bivariate random vectors, we match Serfling and Xiao's proposition \cite{xiao07}.\\

We may define a transport $T:\mathbb{R}^d\rightarrow\mathbb{R}^d$ between two measures $\mu$ and $\nu$ defined on $\mathbb{R}^d$.
\begin{definition}
The pushforward measure of $\mu$ through $T$ is the measure denoted by $T\#\mu$ satisfying 
\begin{equation}
T\#\mu(B) = \mu(T^{-1}(B)) \text{ for every Borel subset $B$ of $\mathbb{R}^d$}
\label{transportdef}
\end{equation}
$T$ is said to be a transport map between $\mu$ and $\nu$ if $T\#\mu=\nu$. In the following, we will call $\mu$ the source measure and $\nu$ the target measure.
\end{definition}

There exist many ways of transporting a measure onto another one. Let us mention for example the transport of Rosenblatt we just mentioned or the transport of Moser \cite{villani09}.\\
The transport that has received the most attention is undoubtedly optimal transport. Its first formulation goes back to 1781 by Monge. More recently, it was in particular studied by Gangbo, McCann, Villani \cite{villani04} \cite{villani09} \cite{mccann95}. In its modern formulation, an optimal transport minimizes a cost function amongst any possible transports.\\
These transports were used by Easton and McCulloch \cite{easton90} in order to generalize the Q-Q plots for multivariate data, a graphical tool close to L-moments that especially shows how far two random samples are apart.\\

However, it is often difficult to have closed forms of the solution of the minimization problem issued from the optimal transport for two arbitrary measures. This is the reason of the following construction of a multivariate quantile.\\
Let $\mathcal{N}_d$ be the canonical Gaussian measure on $\mathbb{R}^d$. The mapping $Q_0:[0;1]^d\rightarrow \mathbb{R}^d$ defined through
\begin{equation}
Q_0(t_1,...,t_d) = \left(\begin{array}{c}\mathcal{N}_1^{-1}(t_1)\\ \vdots \\ \mathcal{N}_1^{-1}(t_d)\end{array}\right)
\end{equation}
transports the uniform measure $unif$ on $[0;1]^d$ onto $\mathcal{N}_d$ (it is actually an optimal transport for a quadratic cost). This quantile (or transport) provides the reference measure $\mathcal{N}_d$.\\

Turning back to the extension of the univariate case, consider $\mu=\mathcal{N}_d$ and $\nu$ any measure on $\mathbb{R}^d$. With $T$ defined as in \ref{transportdef}, we may define a transport from the uniform measure on $[0;1]^d$ onto the measure $\nu$ on $\mathbb{R}^d$ by
\begin{equation}
Q:=T\circ Q_0.
\end{equation}
$Q$ (which is a transport from $unif$ to $\nu$) is a natural extension of the quantile function defined from $[0;1]$ equipped with the uniform measure onto $\mathbb{R}$ equipped with a given measure.\\
Clearly, the intermediate Gaussian measure can be skipped and a quantile may be defined directly from $[0;1]^d$ onto $\mathbb{R}^d$ with the respective measures $unif$ and $\nu$.Indeed, we will define transports from $[0;1]^d$ equipped with $unif$ onto $[0;1]^d$ equipped with a given copula; see Section \ref{section_copula}.\\
The interest in the intermediate (or reference) Gaussian measure $\mu$ lies in the fact that a transport $T$ from $\mu$ onto a measure $\nu$ will be easy to define when $\nu$ belongs to specific classes of multivariate distributions with rotational parameters. Note that the transport $T$ need not be optimal for some cost.\\

We will concentrate our attention on models close to elliptical distributions. Let us recall that elliptical distributions are parametrized by the existence of a scatter matrix $\Sigma$, a location vector $m$ and a radial scalar random variable $R\in\mathbb{R}_+$. In fact, $X\in\mathbb{R}^d$ follows an elliptical distribution if and only if
\begin{equation*}
X\eqlaw m + R\Sigma^{1/2}U
\end{equation*}
with $U$ uniform over $S_{d-1}(0,1)$, the sphere of center zero and radius $1$ and $R$ independent of $U$.\\
Even if, to our knowledge, there are no tractable closed forms for the optimal transport of the uniform on $[0;1]^d$ (or even of the standard Gaussian) onto an elliptical distribution, we can define a family of models close to the elliptical ones that contains spherical distributions with an explicit quantile. This allows to build estimators based on a multivariate method of L-moments for the scatter matrix and the mean parameters of this family.\\
The price to pay for using optimal transports is to consider models adapted to this approach. A natural way to work with such quantiles is then to define models through their quantile function, instead of the classical density function. Sei proposed \cite{sei11} to define models through their transport onto a standard multivariate Gaussian. Such models have desirable properties, in particular the ease to describe the independence of marginals and the concavity of their log-likelihood. In a similar desire to define non-Gaussian distributions easy to manipulate in the context of linear models, Box and Cox used a particular form of this transport as well \cite{boxcox64}.\\

Let us now introduce some notation. In the following, we will consider a random variable or vector $X$ with measure $\nu$ and $\eqlaw$ means the equality in distribution. The scalar product between $x$ and $y$ in $\mathbb{R}^d$ will be noted $x.y$ or $\langle x,y\rangle$.\\

\section{Definition of multivariate L-moments and examples}

\subsection{General definition of multivariate L-moments}

Let $X$ be a random vector in $\mathbb{R}^d$. We wish to exploit the representation given by the equation (\ref{eq:iintlmom}) in order to define multivariate L-moments. Recall that we chose quantiles as mappings between $[0;1]^d$ and $\mathbb{R}^d$.\\

We explicit a polynomial orthogonal basis on $[0;1]^d$. Let $\alpha=(i_1,...,i_d)\in\mathbb{N}^d$ be a multi-index and $L_{\alpha}(t_1,...,t_d) = \prod_{k=1}^d L_{i_k}(t_k)$ (where the $L_{i_k}$'s are univariate Legendre polynomials defined by equation \ref{legendre_pol}) the natural multivariate extension of the Legendre polynomials. Indeed, it holds
\begin{lemma}
The $L_{\alpha}$ family is orthogonal and complete in the Hilbert space $L^2([0;1]^d,\mathbb{R})$ equipped with the usual scalar product :
\begin{equation}
\forall f,g\in L^2([0;1]^d), \ \ \  \langle f,g\rangle=\int_{[0;1]^d} f(u).g(u)du
\end{equation}
\label{legendre}
\end{lemma}
\begin{proof}
The orthogonality is straightforward since if $\alpha=(i_1,...,i_d)\neq\alpha'=(i_1',...,i_d')$, there exists a subindex $1\leq k\leq d$ such that $i_k\ne i_k'$ and 
\begin{equation}
\int_{[0;1]^d} L_{\alpha}(t_1,...,t_d)L_{\alpha'}(t_1,...,t_d)dt_1...dt_d = \prod_{j=1}^d \int_0^1 L_{i_j}(t_j)L_{i_j'}(t_j)dt_j   =0
\end{equation}
thanks to the orthogonality of $L_{i_k'}$ and $L_{i_k}$ in $L_2([0;1],\mathbb{R})$.\\
The univariate Legendre polynomials define an orthogonal basis for the space of polynomials denoted by $\mathbb{R}[X]$. Hence, for all $k$, there exists $c_1,...,c_k\in\mathbb{R}$ such that $X^k = \sum_{i=1}^k c_iL_i(X)$. Thus for all $k_1,...,k_d$, there exists $c_{11},...,c_{1k},...,c_{d1},...,c_{dk}\in\mathbb{R}$ such that 
\begin{equation*}
\prod_{j=1}^d X_j^{k_j} = \prod_{j=1}^d \left( \sum_{i=1}^{k_j} c_{ji} L_i(X) \right).
\end{equation*}
We deduce that $(L_{\alpha})$ is an orthogonal basis of the space of polynomial with $d$ indices $\mathbb{R}[X_1,...,X_d]$. It remains to prove that $\mathbb{R}[X_1,...,X_d]$ is dense in $L_2([0;1]^d,\mathbb{R})$.\\
For this purpose, let $f\in L_2([0;1]^d,\mathbb{R})$. We define a test function $\varphi\in C^0([0;1]^d,\mathbb{R})$ defined for $x\in[0;1]^d$
\begin{equation*}
\varphi(x) = \left\{\begin{array}{ll}
e^{-\frac{1}{1-\|x\|^2}} &\text{   if   } \|x\|<1\\
0  & \text{   if    } \|x\|=1
\end{array}\right.
\end{equation*}
with $\|x\|=\sqrt{\sum_{i=1}^d x_i^2}$.\\
Let $n$ be an integer greater than zero and 
\begin{equation*}
f_n(x) = \frac{1}{\int_{\mathbb{R}^d}\varphi(x)dx} \int_{[\mathbb{R}^d} \frac{1}{n^d}f(x-y)\varphi(\frac{y}{n})\mathds{1}_{x-y\in[0;1]^d}dy
\end{equation*}
Then for all $n> 0, f_n\in C^0([0;1]^d,\mathbb{R}^d)$ and $f_n\rightarrow f$ in $L_2([0;1]^d,\mathbb{R})$. Indeed, by noting $a=\int_{\mathbb{R}^d}\varphi(x)dx$ for $x\in[0;1]^d$
\begin{eqnarray*}
f_n(x)-f(x) &=& \frac{1}{a} \int_{[\mathbb{R}^d} (f(x-y)-f(x))\frac{1}{n^d}\varphi(\frac{y}{n})\mathds{1}_{x-y\in[0;1]^d}dy\\
&=& \frac{1}{a} \int_{[\mathbb{R}^d} (f(x-ny)-f(x))\varphi(y)\mathds{1}_{x-ny\in[0;1]^d}dy
\end{eqnarray*}
Furthermore 
\begin{equation*}
\|f(x-ny)-f(x)\|^2\mathds{1}_{x-ny\in[0;1]^d}\varphi(y)^2\leq 2\varphi(y)^2\int_{[0;1]^d}f(y)^2dy = \|f\|^2_{L_2}\varphi(y)^2,
\end{equation*}
Then as for any $y\in\mathbb{R}^d$, $\|f(x-ny)-f(x)\|^2\mathds{1}_{x-ny\in[0;1]^d}\rightarrow 0$ when $n\rightarrow\infty$; we apply the dominated convergence theorem to show that 
$\|f_n(x)-f(x)\|^2 \rightarrow 0$ for any $x\in[0;1]^d$. In the same way, as 
\begin{equation*}
\|f_n(x)-f(x)\|^2 \leq 2\int_{[0;1]^d}f(y)^2dy
\end{equation*}
We prove by a second application of the dominated convergence theorem that $f_n\rightarrow f$ in $L_2([0;1]^d,\mathbb{R})$.\\
Let $\epsilon>0$. We can thus find $N>0$ such that 
\begin{equation*}
\|f-f_N\|_{L_2}< \epsilon
\end{equation*}
Hence, as $f_N\in C^0([0;1]^d,\mathbb{R}^d)$, by Stone-Weierstrass Theorem (see for example Rudin Theorem 5.8 \cite{rudin91}), there exists $g\in\mathbb{R}[X_1,...,X_d]$ such that :
\begin{equation*}
\|f_N-g\|_{\infty}< \epsilon.
\end{equation*}
Then $\|f-g\|_{L_2} < \|f-f_N\|_{L_2} +\|f_N-g\|_{L_2}  < \epsilon + \|f_N-g\|_{\infty} < 2\epsilon$. We conclude that $\mathbb{R}[X_1,...,X_d]$ is dense in $L_2([0;1]^d,\mathbb{R})$ which proves that $(L_{\alpha})_{\alpha\in\mathbb{N}_*^d}$ is complete.
\end{proof}

We can finally define the multivariate L-moments.
\begin{definition}
\label{def_lmom}
Let $Q:[0;1]^d\rightarrow \mathbb{R}^d$ be a transport between the uniform distribution on $[0;1]^d$ and $\nu$.
Then, if $\mathbb{E}[\|X\|]<\infty$, the L-moment $\lambda_{\alpha}$ of multi-index $\alpha$ associated to the transport $Q$ are defined by : 
\begin{equation}
\label{mlmoments}
\lambda_{\alpha} := \int_{[0;1]^d} Q(t_1,...,t_d) L_\alpha(t_1,...,t_d)dt_1...dt_d \in \mathbb{R}^d.
\end{equation}
\end{definition}
With this definition, there are as many L-moments as ways to transport $unif$ onto $\nu$. The hypothesis of finite expectation guarantees the existence of all L-moments : 
\begin{align*}
\left\| \int_{[0;1]^d} Q(t_1,...,t_d) L_\alpha(t_1,...,t_d)dt_1...dt_d\right\|
&\leq  \left(\sup_{t\in [0;1]^d} |L_\alpha(t)| \right)  \int_{[0;1]^d} \left\| Q(t_1,...,t_d)\right\|dt_1...dt_d\\
&\leq \int_{[0;1]^d} \| x\| dF(x) <\infty.
\end{align*}

\begin{remark}
Given the degree $\delta$ of $\alpha=(i_1,...,i_d)$ that we define by $\delta = \sum_{k=1}^d (i_k-1)+1$, we may define all L-moments with degree $\delta$, each one associated with a given corresponding $\alpha$ leading to the same $\delta$.\\
For example, the L-moment of degree $1$ is 
\begin{equation}
\label{eq212}
\lambda_1 ( =\lambda_{1,1,...,1}) =  \int_{[0;1]^d} Q(t_1,...,t_d) dt_1...dt_d = \mathbb{E}[X].
\end{equation}
The L-moments of degree 2 can be grouped in a matrix :
\begin{equation}
 \Lambda_2 = \left[\int_{[0;1]^d} Q_i(t_1,...,t_d)(2t_j-1) dt_1...dt_d \right]_{1\leq i,j\leq d}.
\end{equation}
In equation \ref{eq212} we noted $Q(t_1,...,t_d) = \left( \begin{array}{cc} Q_1(t_1,...,t_d) \\ \vdots \\ Q_d(t_1,...,t_d) \end{array}\right)$.
\end{remark}

\begin{prop}
Let $\nu$ and $\nu'$ be two Borel probability measures. We suppose that $Q$ and $Q'$ respectively transport $unif$ onto $\nu$ and $\nu'$.\\
Assume that $Q$ and $Q'$ have same multivariate L-moments $(\lambda_{\alpha})_{\alpha\in\mathbb{N}_*^d}$ given by the equation (\ref{mlmoments}).\\
Then $\nu=\nu'$. Moreover :
\begin{equation}
Q(t_1,...,t_d) = \sum_{(i_1,...,i_d)\in \mathbb{N}_*^d}\left(\prod_{k=1}^d (2i_k+1)\right)   L_{(i_1,...,i_d)}(t_1,...,t_d) \lambda_{(i_1,...,i_d)}\in\mathbb{R}^d
\end{equation}
\end{prop}
\begin{proof}
We have to prove that if $Q$ and $Q'$ are two transports coming from $\nu$ and $\nu'$ such that all their L-moments coincide, $\nu=\nu'$.\\
We denote by $\lambda_{\alpha}$ and $\lambda'_{\alpha}$ their respective L-moments of multi-index $\alpha$.\\
As the Legendre family is orthogonal and complete in $L_2([0;1]^d,\mathbb{R})$, we can decompose each component of $Q$ :
\begin{eqnarray*}
Q(t_1,...,t_d) &=& \sum_{\alpha\in \mathbb{N}_d} \frac{\langle Q,L_{\alpha}\rangle_{L_2}}{\langle L_{\alpha},L_{\alpha}\rangle_{L_2}} L_{\alpha}(t_1,...,t_d)\\
&=& \sum_{\alpha\in \mathbb{N}_*^d} \left(\prod_{k=1}^d (2i_k+1)\right)\lambda_{\alpha}L_{\alpha}(t_1,...,t_d)
\end{eqnarray*}
because for $\alpha = (i_1,...,i_d)\in\mathbb{N}_{*}^d$
\begin{equation*}
\int_{[0;1]^d} L_{\alpha}(t_1,...,t_d)^2dt_1...dt_d = \prod_{k=1}^d ||L_{i_k}||^2_{L_2([0;1])} = \prod_{k=1}^d \frac{1}{2i_k+1}.
\end{equation*}
By the same reasoning, we get 
\begin{equation*}
Q'(t_1,...,t_d) = \sum_{\alpha\in \mathbb{N}_*^d} \left(\prod_{k=1}^d (2i_k+1)\right)\lambda'_{\alpha}L_{\alpha}(t_1,...,t_d).
\end{equation*}
We conclude that $Q=Q'$ and $\nu=\nu'$ by hypothesis.
\end{proof}

\subsection{L-moments ratios}
Let us note $\lambda_r(X)$ the r-th univariate L-moment of the random variable $X$ and $(b_1,...,b_d)$ the canonical basis of $\mathbb{R}^d$. Let us decompose the vector $\lambda_{\alpha} $ into
\begin{equation*}
\lambda_{\alpha} = \myvec{\langle \lambda_{\alpha},b_1\rangle }{\langle \lambda_{\alpha},b_d\rangle} \in\mathbb{R}^d.
\end{equation*}
\begin{definition}
\label{def_ratio}
As for univariate L-moments, we can define normalized ratios of L-moments for any multi-index $\alpha\in\mathbb{N}^d$ different from (1,\dots,1) by : 
\begin{equation}
\tau_{\alpha} = \myvec{\langle \tau_{\alpha},b_1\rangle}{\langle \tau_{\alpha},b_d\rangle} = \myvec{\frac{\langle \lambda_{\alpha},b_1\rangle}{\lambda_2(X_1)}}{\frac{\langle \lambda_{\alpha},b_d\rangle}{\lambda_2(X_d)}}.
\end{equation}
with $\lambda_2(X_i)$ denoting the univariate second L-moment related to $X_i$.
\end{definition}
This definition is guided by the following inequality : 
\begin{prop}
For all $\alpha\in\mathbb{N}_*^d$ different from (1,\dots,1), we have : 
\begin{equation}
|\langle \tau_{\alpha},e_i\rangle|\leq 2;
\end{equation}
Moreover, if $\alpha=(i_1, ..., i_d)$ with $i_j=2$ and $i_k=1$ for all $k\neq j$, let  $U=(U_1,...,U_d)^T$ be a uniform random vector on $[0;1]^d$ and $U_{-j} =(U_1,...,U_{j-1},U_{j+1},...,U_d)^T$ and $V =\mathbb{E}_{U_{-j}}[Q_i(U)]$.\\
Then
\begin{equation}
|\langle \tau_{\alpha},b_i\rangle|\leq \frac{\lambda_2(V)}{\lambda_2(X_i)}
\end{equation}
\label{prop_ratio}
\end{prop}

\begin{proof}
Let $y\in\mathbb{R}$. Then as $\alpha\ne (1,\dots,1)$,
\begin{align*}
\langle \lambda_{\alpha},b_i\rangle &= \int_{[0;1]^d} Q_i(t_1,...,t_d)L_{\alpha}(t) dt_1...dt_d \\
&= \int_{[0;1]^d} (Q_i(t_1,...,t_d)-y)L_{\alpha}(t) dt_1...dt_d.
\end{align*}
As $|L_{\alpha}(t)|\leq 1$ for all $t\in[0;1]^d$, by definition of the transport, it holds
\begin{align*}
|\langle \lambda_{\alpha},b_i\rangle| &\leq \int_{[0;1]^d} \left|Q_i(t_1,...,t_d)-y\right|dt_1...dt_d\\
&\leq \int_0^1 \int_0^1 \left|x_i-y\right|dF_i(x_i)dF_i(y)\\
&\leq \mathbb{E}_{X_i\eqlaw Y_i}[|X_i-Y_i|] = 2\lambda_2(X_i).
\end{align*}
This proves the first assertion. The second is inspired from the proposition 4 of \cite{xiao07}.\\
As the degree of $\alpha$ is 2, there exists $1\leq j\leq d$ such that
\begin{equation*}
\langle \lambda_{\alpha},b_i\rangle = \int\left(\int Q_i(t_1,...,t_d)L_2(t_j)dt_j\right)dt_1...dt_{j-1}dt_{j+1}...dt_d.
\end{equation*}
We note $U_{-i} =(U_1,...,U_{j-1},U_{j+1},...,U_d)'$, $V =\mathbb{E}_{U_{-j}}[Q_i(U)]$ and $W=U_j$. Then by noting 
\begin{eqnarray*}
\langle \lambda_{\alpha},b_i\rangle &=& \mathbb{E}[VL_2(W)]\\
&=&2\mathbb{E}[VW] - \mathbb{E}[V]\\
&=&2Cov(V,W)
\end{eqnarray*}
where $V$ and $W$ are two random variables of finite expectation and covariance. Then, Hoeffding lemma quoted in \cite{lehmann66} gives us :
\begin{equation*}
Cov(V,W) = \int\int \left[F_{V,W}(v,w) - F_V(v)F_W(w)\right]dvdw
\end{equation*}
Moreover, the well-known Fr\'{e}chet bounds assert that for any $v,w$
\begin{equation*}
\max(F_W(w)+F_V(v)-1,0)\leq F_{v,W}(v,w) \leq \min(F_V(v),F_W(w)).
\end{equation*}
Since $W$ is uniform on $[0;1]$
\begin{equation*}
Cov(V,W) \leq \int\int \left[ \min(F_V(v),w) - F_V(v)w\right]dvdw.
\end{equation*}
Furthermore
\begin{equation*}
Cov(V,F_V(V)) = \int\int \left[ \min(F_V(v),w) - F_V(v)w\right]dvdw.
\end{equation*}
We conclude that
\begin{equation*}
Cov(V,W) \leq Cov(V,F_V(V)).
\end{equation*}
Now, using $\max(a+b-1,0)-ab=-(\min(1-a,b)-(1-a)b)$ along with the Fr\'{e}chet bound, a similar reasoning leads to 
\begin{equation*}
Cov(V,W) \geq -Cov(V,F_V(V)).
\end{equation*}
Remarking that $2Cov(V,F_V(V))=\lambda_2(V)$, we obtain
\begin{equation*}
|\langle \lambda_{\alpha},b_i\rangle|\leq \lambda_2(V).
\end{equation*}

\end{proof}

\begin{remark}
The inequality in the previous Proposition is probably not optimal but has the advantage of some generality. As we will see later, if we choose the particular bivariate Rosenblatt transport, it holds $|\langle \tau_{\alpha},b_i\rangle|\leq 1$ for $\alpha=(1,2)$ or $\alpha=(2,1)$.
\end{remark}

\subsection{Compatibility with univariate L-moments}

The definition which we adopted for the definition of general L-moments is compatible with the similar one in dimension 1 since the univariate quantile is a transport.
\begin{definition}
Let $\nu$ be a real probability measure. The quantile is the generalized inverse of the distribution function :
\begin{equation}
Q(t) = \inf \{ x\in\mathbb{R} \text{ s.t. }  \nu((-\infty;x]) \geq t \}.
\end{equation}
\end{definition}

\begin{prop}
If we denote by $\mu$ the uniform measure on $[0;1]$, then $Q\#\mu = \nu$ i.e. $Q(U)\eqlaw X$ if $U$ denotes the uniform law on $[0;1]$, and $X$ denotes the random variable associated to $\nu$.
\end{prop}
\begin{proof}
Let $x\in\mathbb{R}$. We denote by $F$ the cdf of $X$ and by $A_t$ the event
\begin{equation*}
A_t = \left\{x\in\mathbb{R} \text{   s.t.    } F(x)\geq t\right\}
\end{equation*}
We then have $Q(t) = \inf A_t$. We wish to prove :
\begin{equation}
\left\{ t\in[0;1]\text{     s.t.     } Q(t)\leq x\right\} = \left\{ t\in[0;1]\text{      s.t.     }    t\leq F(x)\right\} 
\label{equivalence}
\end{equation}
We temporarily admit this assertion. Then 
\begin{eqnarray*}
\mathbb{P}[Q(U)\leq x] &=& \mathbb{P}[U\leq F(x)]\\
 &=&F(x)
\end{eqnarray*}
which ends the proof. It remains to prove \ref{equivalence}.\\
First, the definition of $Q$ gives us 
\begin{equation*}
\left\{ t\leq F(x)\right\} \Rightarrow \left\{ x\in A_t\right\}\Rightarrow \left\{ Q(t)\leq x\right\}
\end{equation*}
Secondly, let $t$ be such that $Q(t)\leq x$. Then by monotony of $F$, $F(Q(t))\leq F(x)$. We then claim that
\begin{equation*}
Q(t)\in A_t
\end{equation*}
Indeed, let us suppose the contrary and consider a strictly decreasing sequence $x_n\in A_t$ such that
\begin{equation*}
\lim_{n\rightarrow \infty} x_n = \inf A_t = Q(t).
\end{equation*}
By right continuity of $F$ 
\begin{equation*}
\lim_{n\rightarrow \infty} F(x_n) = F(Q(t))
\end{equation*}
and, on the other hand, by definition of $A_t$,
\begin{equation*}
\lim_{n\rightarrow \infty} F(x_n) \geq t
\end{equation*}
i.e. $Q(t)\in A_t$ wihch contradicts the hypothesis. Then $Q(t)\in A_t$ i.e. $t\leq F(Q(t))$ thus $t\leq F(x)$. We have proved that 
\begin{equation*}
 \left\{ Q(t)\leq x\right\}\Rightarrow \left\{ t\leq F(x)\right\}
\end{equation*}
\end{proof}

Subsequently, if we consider the particular transport defined by the univariate quantile, the L-moments are defined by
\begin{equation}
\lambda_r = \int_0^1 Q(t)L_r(t)dt
\end{equation}
which is the quantile characterization of univariate L-moments.
\begin{remark}
This transport corresponds to a Rosenblatt transport and an optimal transport with respect to a large family of costs (see Proposition \ref{prop5}).
\end{remark}

\subsection{Relation with depth-based quantiles}

With the DOQR paradigm, Sefling related the four following notions:
\begin{itemize}
\item the \textbf{centered quantile function} : a centered multivariate quantile function $Q$ indexed by $u\in B_d$, the unit ball in $\mathbb{R}^d$ such that $x:=Q(u)$ is a centered quantile representation of $x$. $Q(0)$ represents the center of mass or median. This quantile function generates nested contours $\{Q(u) : \|u\|=c\}$ grouping points of the distribution by "distance" to the center of mass.
\item the \textbf{centered rank function} : if the quantile $Q:B_d\rightarrow \mathbb{R}^d$ has an inverse, noted $R:\mathbb{R}^d\rightarrow B_d$, it corresponds to the centered rank function. For each point $x$, $R(x)$ corresponds to the directional rank of $x$.
\item The \textbf{outlyingness function} : the magnitude $O(x) := \|R(x)\|$ defines a measure of the outlyingness of $x$.
\item The \textbf{depth function} : the magnitude $D(x) :=1-O(x)$ provides a center-outward ordering of $x$, higher depth corresponding to higher centrality.
\end{itemize}
With this paradigm, all the depth functions introduced for example in \cite{zuo00} can induce a quantile function (see \cite{serfling10}). Even if the quantile deduced from a depth function is not uniquely defined, the contours associated to the depth are unique.\\

If we note $Q$ the quantile as a transport between the uniform distribution in $[0;1]^d$ and the distribution of interest, then the function 
\begin{equation*}
\tilde{Q} := u\in[-1;1]^d \mapsto Q(\frac{u}{2}-(1/2,...,1/2)^T)
\end{equation*}
correspond to the Serfling's notion of centered quantile for the infinite norm. If $Q$ is invertible, we can therefore introduce a related depth function as
\begin{equation*}
D(x) = 1-2\|Q^{-1}(x)-(1/2,...,1/2)^T\|.
\end{equation*}
This allows us to compare this depth function with respect to the desirable criteria for a depth function enounced in \cite{zuo00} satisfied by classical depth functions such as Tukey's half-space depth function.
\begin{itemize}
\item \textbf{Affine invariance} : the depth of a point $x\in\mathbb{R}^d$ should not depend on the underlying coordinate system. This property is not verified by the depth issued from transport and should be a stake for future works.
\item \textbf{Maximality at center} : the obvious center $Q(1/2,...,1/2)$ is the point of maximal depth
\item \textbf{Monotonicity relative to deepest point} : as the point $x\in\mathbb{R}^d$ moves away from the center of mass, the depth function evaluated on $x$ decreases monotonically. This intuitive property should restrict the transports acceptable for $Q$ to be a quantile. For monotone and Rosenblatt transports introduced in the sequel, this property holds.
\item \textbf{Vanishing at infinity} : the depth of a point $x$ should approach zero as $\|x\|$ approaches infinity. 
\end{itemize}

The quantile function issued from a transport brings moreover indications on the location of the mass of the multivariate distribution of measure $\nu$. Indeed, all intuitive information of a "piece" of the unit cube (centrality, extremality, volume,...) can be transposable to the transported piece of points in $\mathbb{R}^d$. In mathematical terms, if $A$ is Borelian of $[0;1]^d$, it holds :
\begin{equation*}
\nu(Q(A)) = \mu(A)= vol(A)
\end{equation*}

We will now consider in the following two different kinds of transport among many others :
\begin{itemize}
\item the optimal transport
\item the Rosenblatt transport
\end{itemize}

\section{Optimal transport}
\subsection{Formulation of the problem and main results}

Let us consider two measures $\mu$ and $\nu$ respectively defined on $\Omega\subset\mathbb{R}^d$ and $\mathbb{R}^d$. If we define a cost function $c : \Omega\times\Omega \rightarrow \mathbb{R}$, then the problem is to find an application $T$ that transports $\mu$ into $\nu$ and minimizes :
\begin{equation}
\int_{\Omega} c(x,T(x)) d\mu(x).
\end{equation}

The quadratic case $c:(x,y)\mapsto (x-y)^2$ was first studied by Brenier \cite{brenier91}, the generalization to generic costs has been considered, among others, by McCann, Gangbo, Villani \cite{mccann95}\cite{villani04}. Let us give the following theorem for specific convex costs $(x,y)\mapsto c(x,y)=h(x-y)$ :
\begin{theorem}(McCann, Gangbo)\\
Let $h: \mathbb{R}^d\rightarrow \mathbb{R}$ be a convex function, $\mu$ and $\nu$ be two probability measures on $\mathbb{R}^d$. Let us suppose that there exists a transport $T$ such that $\int_{\mathbb{R}^d} h(x-T(x)) d\mu(x)<\infty$. Let us assume that $\mu$ is absolutely continuous with respect to the Lebesgue measure. \\
Then, there exists a unique transport $T$ from $\mu$ to $\nu$ that minimizes the cost $\int_{\Omega} h(x-T(x)) d\mu(x)$ determined $d\mu$-almost everywhere and characterized by a function $\phi$ :
\begin{equation}
T(x) = x - \nabla h^*\left(\nabla\phi(x)\right)
\end{equation}
where $h^*$ is the Legendre transform of $h$. 
\begin{equation*}
h^*(y) = \sup_{x\in\mathbb{R}^d} \langle x,y\rangle -h(x).
\end{equation*}
The function $\phi$ is $d\mu$-a.s. unique up to an arbitrary additive constant.
\label{gangbo}
\end{theorem}
\begin{proof}
We apply Theorem 2.44 of \cite{villani04}.
\end{proof}

\begin{remark}
If we consider the quadratic case $h(x-y) =(x-y)^2$, the above existence theorem is equivalent to the existence of another function (that will be called potential function) $\varphi:=x\mapsto\|x\|^2 - \phi(x)$ which is convex such that the optimal transport is $T = \nabla\varphi$. We can observe a refinement of this case in the following Proposition \ref{prop_brenier}.
\end{remark}

For the definition of a multivariate quantile, $\mu$ is the uniform measure on $\Omega=[0;1]^d$ and $\nu$ is the measure of a random vector $X$ of interest. The corresponding transport will be denoted by $Q$.\\
As $\mu$ is absolutely continuous with respect to the Lebesgue measure, the remaining assumption in Theorem \ref{gangbo} is the existence of a transport $Q$ such that $Q\#\mu=\nu$ and $\int_{\Omega} h(Q(u)-u) du<\infty$.\\
We can remove this limitation by considering source measures $\mu$ that give no mass to "small sets". To make the term "small set" more precise, we use the Hausdorff dimension.
\begin{definition}
Let $E$ be a metric space. If $S\subset  E$ and $p\geq 0$, the p-dimensional Hausdorff content of $S$ is defined by :
\begin{equation*}
C_d(S) = \inf \left\{  \sum_i r_i^p\text{  such that there is a cover of $S$ by balls with radii $r_i>0$} \right\}.
\end{equation*}
Then, the Hausdorff dimension of $E$ is given by :
\begin{equation}
dim(E) := \inf\left\{ d\geq 0 \text{   such that $C_d(E) = 0$}\right\}
\end{equation}
\end{definition}

\begin{prop}(McCann/Brenier's Theorem)\\
\label{prop_brenier}
Let $\mu, \nu$ be two probability measures on $\mathbb{R}^d$, such that $\mu$ does not give mass to sets of Hausdorff dimension at most $d-1$. Then, there is exactly one measurable map $T$ such that $T\#\mu=\nu$ and $T=\nabla\varphi$ for some convex function $\varphi$, in the sense that any two such maps coincide $d\mu$-almost everywhere.
\end{prop}
\begin{proof}
Theorem 2.32 of \cite{villani04}
\end{proof}

\begin{remark}
When $\mu$ is the uniform measure on $[0;1]^d$, Proposition \ref{prop_brenier} holds for any $\nu$.
\end{remark}

The gradient of convex potentials are called monotone by analogy with the univariate case. We can see this gradient as the solution of a potential differential equation. By abuse of language, we will refer at this transport as monotone transport in the sequel.
\begin{remark}
Let us suppose $\mu$ and $\nu$ admit densities with respect to the Lebesgue measure respectively denoted by $p$ and $q$. Proposition \ref{prop_brenier} provides a mapping $\nabla\varphi$ such that for all test functions $a$ in $C^{\infty}$ with a compact support :
\begin{equation*}
\int a(y)q(y)dy = \int a(\nabla \varphi(x))p(x)dx.
\end{equation*}
Let us assume furthermore that $\nabla\varphi$ is $C^1$ and bijective. We can then perform the change of variables $y=\nabla\varphi(x)$ on the left-hand side of the previous equality : 
\begin{equation*}
\int a(y)q(y)dy = \int a(\nabla \varphi(x)) q(\nabla \varphi(x)) \det(\nabla^2\varphi(x)) dx.
\end{equation*}
Since the function $a$ is arbitrary, we get :
\begin{equation}
\label{monge_ampere}
p(x) = q(\nabla \phi(x))) \det(\nabla^2 \phi(x)).
\end{equation}
This is a particular case of the general Monge-Amp\`{e}re equation
\begin{equation*}
\det(\nabla^2\varphi(x)) = F(x,\varphi(x),\nabla\varphi(x)).
\end{equation*}

\end{remark}

\subsection{Optimal transport in dimension 1}

The natural order of the real line implies that the quantile is a solution of several transport problems :
\begin{prop}(Optimal transport in dimension $d=1$)\\
\label{prop5}
Let $\mu$ and $\nu$ be two arbitrary measures respectively defined on $[0;1]$ and $\mathbb{R}$ such that $\mu$ gives no mass to atoms. Let $T : [0;1]\rightarrow \mathbb{R}$ be a transport of $\mu$ onto $\nu$. Then for any real convex function $h$ :
\begin{equation}
\int_0^1 h(Q(F(u))-u) du \leq \int_0^1 h(T(u)-u)) du 
\end{equation}
where $Q$ is the generalized quantile of $\nu$ and $F$ the cdf of $\mu$ i.e. $Q\circ F$ is the solution of the univariate transport problem.
\end{prop}
\begin{proof}
We refer to Theorem 6.0 \cite{ambrosio05}.
\end{proof}

This result is particular to the dimension 1. If we plug this result in the definition we gave for multivariate L-moments applied with $d=1$, we obtain the univariate definition of L-moments : 
\begin{equation}
\lambda_r = \int_0^1 Q(t)L_r(t)dt.
\end{equation}
The multivariate L-moments defined with the optimal transport are then compatible with the definition in dimension $d=1$.

\subsection{Examples of monotone transports}
\begin{example}(Univariate Gaussian)\\
Let us consider the univariate Gaussian $\mathcal{N}_{m,\sigma}$ with $m\in\mathbb{R}$ and $\sigma>0$. The potential is then defined up to a constant by :
\begin{equation}
\text{for } t\in[0;1],\ \ \phi_{m,\sigma}(t) = \int_{1/2}^t \left(m + \sigma\mathcal{N}^{-1}(u)\right) du
\end{equation}
where $\mathcal{N}$ is the cumulative distribution function of the standard Gaussian. 

\begin{figure}[h!]
\centering
\includegraphics[height=2.5in]{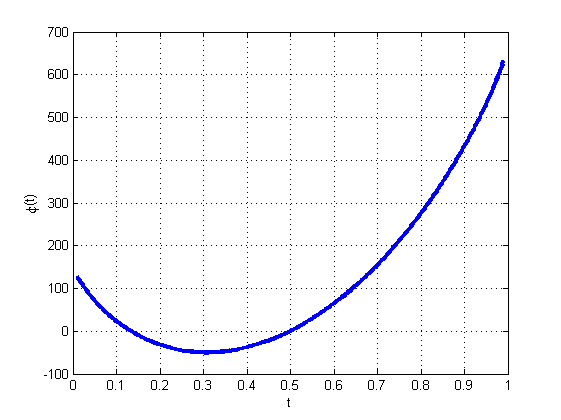}
\caption{Potential function $\phi_{m,\sigma}$ for $m=1$ $\sigma=2$}
\end{figure}

Let us note that the potential is minimum at the cumulative weight $t$ such that $\mathcal{N}_{m,\sigma}^{-1}(t)=0$ and is equal to $0$ at the median.\\
The gradient is simply the quantile
\begin{equation*}
\nabla\phi_{m,\sigma}(t) = m + \sigma\mathcal{N}^{-1}(t) = Q_{\mathcal{N}}(t).
\end{equation*}

\begin{figure}[h!]
\centering
\includegraphics[height=2.5in]{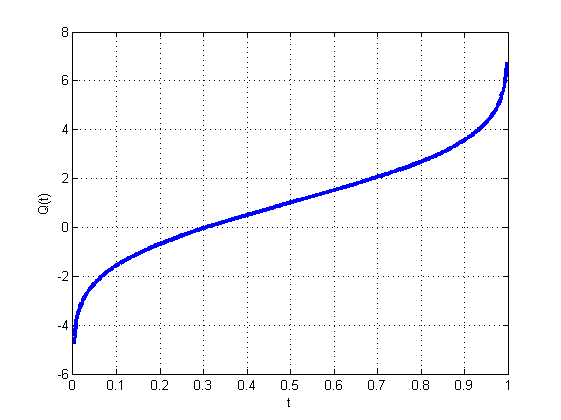}
\caption{Quantile function $\nabla\phi_{m,\sigma}$ for $m=1$ $\sigma=2$}
\end{figure}

If we build the Legendre transform of the potential, we find a dual potential :
\begin{equation*}
\psi_{m,\sigma}(x) =  \sup_{t\in[0;1]} \{xt - \phi_{m,\sigma}(t)\} = x\mathcal{N}\left(\frac{x-m}{\sigma}\right) - \frac{1}{\sigma}\int_m^x x\mathcal{N}'\left(\frac{x-m}{\sigma}\right) dy,
\end{equation*}
and
\begin{equation*}
\nabla\psi_{m,\sigma}(x) = \mathcal{N}\left(\frac{x-m}{\sigma}\right).
\end{equation*}
Finally the Hessian of the dual potential is the density of the Gaussian as expected by the Monge-Amp\`{e}re equation
\begin{equation*}
\nabla^2\psi_{m,\sigma}(x) = \frac{1}{\sigma}\mathcal{N}'\left(\frac{x-m}{\sigma}\right) =  \frac{1}{\sigma\sqrt{2\pi}}\exp\left(-\frac{(x-m)^2}{2\sigma^2}\right).
\end{equation*}

\end{example}

\begin{example}(Independent coordinates)\\
\label{ex_independent}
For a vector of independent marginals $(X_1, ..., X_d)$, the optimal transport is easily obtained since it is the concatenation of each marginal univariate quantile: 
\begin{equation}
Q(t_1,...,t_d) = \left(\begin{array}{cc} Q_1(t_1)\\ \vdots\\ Q_d(t_d) \end{array}\right)
\end{equation}
if $Q_1, ..., Q_d$ are the respective quantiles of $X_1,..., X_d$.\\
Indeed, it is obvious that the mapping $Q$ defined above transports the uniform measure on $[0;1]^d$ into the distribution of $(X_1, ..., X_d)$.\\
Furthermore, as $Q_1, ..., Q_d$ are univariate transports, they are gradients of convex functions that can be denoted by respective potentials $\phi_1,...,\phi_d : [0;1] \rightarrow \mathbb{R}$. Then, if we build the potential :
\begin{equation}
\phi(t_1,..., t_d) = \phi_1(t_1) + \dots + \phi_d(t_d),
\end{equation}
we remark that $\nabla \phi = Q$ and $\phi$ is convex because each $\phi_i$ is convex.
\end{example}

\begin{example}(Max-Copula)\\
\label{maxcopula}
Let us define a 2-dimensional potential for $u,v\in[0;1]^2$: 
\begin{equation}
\phi(u,v) = \frac{1}{4}(u+v)^2.
\end{equation}

\begin{figure}[h!]
\centering
\includegraphics[height=2.5in]{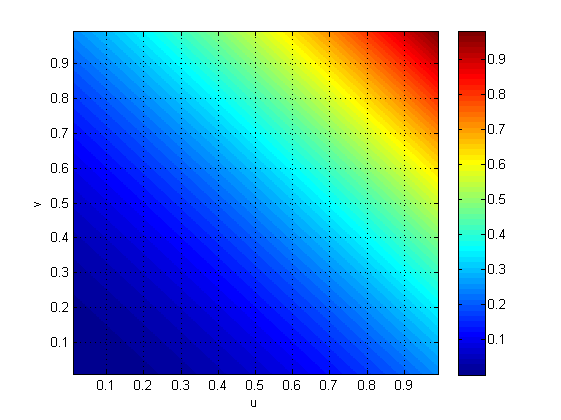}
\caption{Color levels of the potential function $\phi$}
\end{figure}
$\phi$ is convex (but not strictly convex) and derivable almost everywhere; the associated transport is
\begin{equation*}
T(u,v) = \nabla \phi(u,v) = \myvecz{\frac{u+v}{2}}{\frac{u+v}{2}}
\end{equation*}
T then transports the uniform distribution on $[0;1]^2$ into the distribution defined by the cdf $F(u,v) = \min(u,v)$.
\end{example}
The distribution considered in the previous example corresponds to the max-copula. A copula induces a distribution defined on $[0;1]^2$ with uniform margins and is a measure of the dependence for a bivariate random vector.\\

\section{L-moments issued from the monotone transport}
From now on, the notion of optimal transport will uniquely refer to the monotone case.

\subsection{Monotone transport from the uniform distribution on $[0;1]^d$}
Let $X\in\mathbb{R}^d$ be a random vector. According to Brenier's Theorem, there exists a potential $\varphi:[0;1]^d\rightarrow \mathbb{R}^d$ such that 
\begin{equation*}
\nabla\varphi(U) \eqlaw X.
\end{equation*}
The $\alpha$-th multivariate L-moment associated to this potential is
\begin{equation}
\lambda_{\alpha} = \int_{[0;1]^d} \nabla\varphi(t)L_{\alpha}(t)dt.
\end{equation}

We keep the property of invariance with respect to translation and equivariance with respect to dilatation coming from the univariate L-moments
\begin{prop}
Let $X$ be a random vector in $\mathbb{R}^d$ and $\lambda_{\alpha}(X)$ its associated L-moments such that
\begin{equation}
\lambda_{\alpha}(X) = \int_{[0;1]^d} \nabla\varphi(t)L_{\alpha}(t)dt
\end{equation}
with $\nabla\varphi$ the transport from the uniform on $[0;1]^d$ onto $X$ and $\varphi$ convex.\\
Let $m\in\mathbb{R}^d$ and $\sigma>0$. Then
\begin{equation}
\lambda_{\alpha}(m+\sigma X) = \sigma\lambda_{\alpha} + m\mathds{1}_{\alpha=(1...1)}
\end{equation}
\label{equivariance1}
\end{prop}
\begin{proof}
Let $\psi:x\mapsto \sigma\varphi(x) + \langle x,m\rangle$. Then $\psi$ is convex and $\nabla\psi(X)=\sigma X + m$.\\
$\nabla\psi$ is then the monotone transport from the uniform distribution on $[0;1]^d$ onto the distribution of $\sigma X+m$.
\end{proof}
We do not have a general property dealing with rotation with this transport. This is strongly due to the bad behavior of the unit square through this transformation. We will present later Hermite L-moments that partially fill this deficiency.

\subsection{Monotone transport for copulas}
\label{section_copula}
Let $X$ be a bivariate vector of cdf denoted by $H$. We can build a transport of the bivariate uniform distribution on $[0;1]^2$ into $X$ through the composition of the transport of the copula of $X$ with the transport of the marginals. The reason of this construction is that the copula function is well adapted to the unit square $[0;1]^d$.\\

Let us first present the definition of a copula and Sklar's Theorem.
\begin{definition}
A copula is a function $C : [0;1]^2\rightarrow [0;1]$ with the following properties :
\begin{itemize}
\item C is 2-increasing i.e. for all $u_1\leq u_2\in[0;1]$ and $v_1\leq v_2\in [0;1]$ : 
\begin{equation*}
C(u_2,v_2) - C(u_2,v_1) - C(u_1,v_2) + C(u_1,v_1) \geq 0
\end{equation*}
\item for $u,v\in[0;1]$ :
\begin{equation*}
C(u,1) = u \text{ , } C(u,0) = 0 \text{         and        } C(1,v) = v\text{ , } C(0,v) = 0 
\end{equation*}
\end{itemize}
\end{definition} 

\begin{theorem}(Sklar's theorem)\\
\label{sklar_th}
Let $H$ be a joint distribution with margins $F$ and $G$. Then there exists a copula $C$ such that for all $x,y\in\bar{\mathbb{R}}=\mathbb{R}\cup\{-\infty,+\infty\}$ :
\begin{equation}
H(x,y) = C(F(x),G(y))
\end{equation}
C is uniquely defined on $F(\bar{\mathbb{R}})\times G(\bar{\mathbb{R}})$.\\
Conversely, if $C$ is a copula and $F$ and $G$ are distribution functions, then $H$ defined by the above equation is a joint distribution function with margins $F$ and $G$.
\end{theorem}
\begin{proof}
see Theorem 2.3.3 of Nelsen \cite{nelsen06}.
\end{proof}

Let now $C$ be a copula associated to the bivariate vector $X$. Then, using the previous Theorem \ref{sklar_th}, $C$ is the joint distribution function of a bivariate vector $W$ with uniform margins on $[0;1]$. Let $Q_C$ be the optimal transport of $U$ with uniform distribution on $[0;1]^2$ into $W= \myvecz{W_1}{W_2}$. As $W$ and $X$ share the same copula, it is sufficient to transport the margins of $X$ : if $Q_1$ and $Q_2$ transport $W_1$ into $X_1$ and $W_2$ into $X_2$ respectively (we recall that $W_1$ and $W_2$ are uniform such that we naturally choose $Q_1$ and $Q_2$ as the univariate quantiles of $X_1$ and $X_2$) then the function defined for $u,v\in[0;1]$ by
\begin{equation}
Q(u,v) =  \myvecz{Q_1}{Q_2} \circ Q_C (u,v)
\end{equation}
transports $U$ into $X$. To sum up, if we manage to transport a copula, we can easily transport all distributions sharing this copula.\\
We can link the copula function to the potential of Proposition \ref{prop_brenier} :
\begin{lemma}
Let $C$ be a copula and $Q_C = \nabla \phi_C$ the monotone transport between the uniform distribution function and the distribution whose cdf is $C$. Then for all $u,v\in[0;1]$: 
\begin{equation}
C(u,v) = vol\left( \left(\nabla \phi_C\right)^{-1} ([0;u]\times [0;v])\right).
\end{equation}
\end{lemma}
\begin{proof}
Let $W = \myvecz{W_1}{W_2}$ be the distribution of cdf $C$. By definition, we have $W \eqlaw \nabla\phi_C(U)$ ($U$ is uniform on $[0;1]^2$).\\
Then, if $u,v\in[0;1]$
\begin{equation*}
C(u,v) = \mathbb{P}[W_1\leq u, W_2\leq v] = \mathbb{P}[\partial_1 \phi_C(U) \leq u, \partial_2 \phi_C(U)\leq v] = vol\left(\left(\nabla \phi_C\right)^{-1} ([0;u]\times [0;v])\right).
\end{equation*}
\end{proof}

\begin{example}(Independent Copula)\\
The case of the independent copula $\Pi(u,v) = uv$ is straightforward. In that case, the potential is $\phi_\Pi(u,v) = \frac{u^2}{2}+\frac{v^2}{2}$ which gives :
\begin{equation}
Q_\Pi(u,v ) =\nabla\phi_\Pi(u,v) =  \myvecz{u}{v}.
\end{equation}
As (U,V) are uniform independent, $Q_\Pi(u,v)$ have independent margins and its copula is $\Pi$. $\phi_\Pi$ is then the associated potential for the independent copula.\\
The L-moments of the copula's distribution are then for $j,k>0$ :
\begin{equation*}
\lambda_{jk} = \int_{[0;1]^2} \myvecz{u}{v} L_j(u)L_k(v) dudv = \myvecz{\mathds{1}_{k=1} \lambda_j(U)}{\mathds{1}_{j=1} \lambda_k(U)}
\end{equation*}
where $U$ represents a uniform distribution on [0;1] i.e.
\begin{equation*}
\lambda_{11} = \myvecz{\frac{1}{2}} {\frac{1}{2}} \text{  ,  }\lambda_{12} = \myvecz{0} {\frac{1}{6}} \text{  ,  }\lambda_{21} = \myvecz{\frac{1}{6}} {0}\text{  ,  }
\lambda_{jk} = 0 \text{  otherwise}.
\end{equation*}

\end{example}

\begin{example}(Max-Copula, continued)\\
The copula-max $M(u,v) = \min(u,v)$ was treated in Example \ref{maxcopula}. The associated transport is :
\begin{equation}
Q_M(u,v) = \frac{1}{2} \myvecz{u+v}{u+v}.
\end{equation}
The L-moments of the copula's distribution are then for $j,k>0$ :
\begin{equation*}
\lambda_{jk} = \frac{1}{2} \int_{[0;1]^2} \myvecz{u+v}{u+v} L_j(u)L_k(v) dudv = \frac{1}{2} \myvecz{\mathds{1}_{k=1} \lambda_j(U) + \mathds{1}_{j=1} \lambda_k(U)}{\mathds{1}_{k=1} \lambda_j(U) + \mathds{1}_{j=1} \lambda_k(U)}
\end{equation*}
where $U$ represents a uniform distribution on $[0;1]$ i.e.
\begin{equation*}
\lambda_{11} = \myvecz{\frac{1}{2}} {\frac{1}{2}} \text{  ,  }\lambda_{12} = \myvecz{\frac{1}{12}} {\frac{1}{12}} \text{  ,  }\lambda_{21} = \myvecz{\frac{1}{12}} {\frac{1}{12}}\text{  ,  }
\lambda_{jk} = 0 \text{  otherwise}.
\end{equation*}
\end{example}

\begin{example}(Min-Copula)\\
The case of the copula-min $W(u,v) = \max(u+v-1,0)$ can similarly be solved.\\
Let us define the potential $\phi_W(u,v) = \frac{1}{4} (u+1-v)^2$, then for $u,v\in [0;1]$ :
\begin{equation}
Q_W(u,v) = \nabla\phi_W(u,v) = \frac{1}{2} \myvecz{u+1-v}{v+1-u}
\end{equation}
If $U$ and $V$ are uniform and independent $Q_W(U,V)$ has uniform margins that are anti-comonotone i.e. the copula of $Q_W(U,V)$ is $W$.
The L-moments of the copula are then for $j,k>0$ :
\begin{equation*}
\lambda_{jk} = \frac{1}{2} \int_{[0;1]^2} \myvecz{u+1-v}{v+1-u} L_j(u)L_k(v) dudv = \frac{1}{2} \myvecz{\mathds{1}_{k=1} \lambda_j(U) + \mathds{1}_{j=1} (-1)^k\lambda_k(U)}{-\mathds{1}_{k=1} (-1)^j\lambda_j(U) - \mathds{1}_{j=1} \lambda_k(U)}
\end{equation*}
where $U$ represents a uniform distribution on [0;1] i.e.
\begin{equation*}
\lambda_{11} = \myvecz{\frac{1}{2}} {\frac{1}{2}} \text{  ,  }\lambda_{12} = \myvecz{\frac{1}{12}} {-\frac{1}{12}} \text{  ,  }\lambda_{21} = \myvecz{-\frac{1}{12}} {\frac{1}{12}}\text{  ,  }
\lambda_{jk} = 0 \text{  otherwise}.
\end{equation*}

\end{example}

For the sake of simplicity, we have presented copulas in the bivariate setting but multivariate generalizations of copulas and of Sklar's theorem exist. It is then straightforward to adapt the above transport for multivariate random vectors.

\begin{remark}
Even if it is difficult to find the explicit formulation of the monotone transport for classical parametric family of copulas (such as Gumbel or Clayton copula), we can define a copula from its potential.
\end{remark}

\subsection{Monotone transport from the standard Gaussian distribution}
\label{section_gauss}
The major drawback of the uniform law on $[0;1]^d$ is its non-invariance by rotation which is a desirable property in order to more easily compute the monotone transports. For example, the multivariate standard Gaussian distribution appears as a better source measure but any other distribution could also be considered.\\
We propose an alternative transport leading to the following L-moments :
\begin{equation}
\lambda_{\alpha} = \int_{[0;1]^d} T_0 \circ Q_{\mathcal{N}}(t_1,...,t_d)L_\alpha (t_1,...,t_d) dt_1...dt_d
\end{equation}
where $Q_{\mathcal{N}}$ is the transport of the multivariate standard distribution $\mathcal{N}(0,I_d)$ into the uniform one defined by 
\begin{equation*}
Q_{\mathcal{N}}(t_1,..,t_d) = \myvec{\mathcal{N}^{-1}(t_1)}{\mathcal{N}^{-1}(t_d)}
\end{equation*}
and $T_0$ the transport of the considered distribution into the multivariate standard distribution :
\begin{equation}
([0;1],du) \overset{Q_{\mathcal{N}}}{ \rightarrow} (\mathbb{R}^d,d\mathcal{N}) \overset{T_0}{ \rightarrow} (\mathbb{R}^d,d\nu)
\end{equation}

In \cite{sei11}, Sei used the transport from the standard Gaussian in order to define distributions through a convex potential $\varphi$ (actually, he proposed to take the dual potential in the sense of Legendre duality). A useful property of this transport is given by the following lemma
\begin{lemma}
Let $A\in O_d(\mathbb{R})$ be the space of orthogonal matrices (i.e. $AA^T = A^TA=I_d$), $m\in\mathbb{R}^d$, $a\in\mathbb{R}^*$. Let us denote by $\phi$ the potential linked to the random variable such that $\nabla \phi(N)\eqlaw X$ with $N\in\mathbb{R}^d$ a standard Gaussian random vector.\\
 Then the respective potentials related to the vectors $AX$, $aX$ and $X+m$ are $\phi(Ax)$, $a\phi(x)$ and $\phi(x)+m.x$.
\end{lemma}
\begin{proof}
Let $\psi_A(x)=\phi(Ax)$, then $\psi$ is convex and $\nabla \psi_A(x) = A\phi(Ax)$. Furthermore, as $A$ is an orthogonal matrix, $AN\eqlaw N$ which implies $\nabla\psi_A(N)\eqlaw AX$.\\
In the same way, if $\psi_a(x)=a\phi(x)$ and $\psi_m(x)=\phi(x)+m$, we have
\begin{eqnarray*}
\nabla\psi_a(N) &=& a\nabla\phi(N)\eqlaw aX\\
\nabla\psi_m(N) &=& \nabla\phi(N)+m\eqlaw X+m.
\end{eqnarray*}
\end{proof}
Unfortunately, the generalization to all affine transformations is not easy. This is why it is often more convenient to define distributions through their potential function as in Sei's article.

\begin{example}(L-moments of multivariate Gaussian)\\
Let us consider $m\in\mathbb{R}^d$, a positive matrix $A$ and the quadratic potential : 
\begin{equation}
\varphi(x) = m.x + \frac{1}{2} x^TAx \ \ \ \ \ \ \text{  for   } x\in\mathbb{R}^d.
\end{equation}
The transport associated to this potential is :
\begin{equation}
T_0(x) = \nabla\varphi(x) = m+Ax \ \ \ \ \ \ \text{  for   } x\in\mathbb{R}^d.
\end{equation}
Furthermore,  $T_0(\mathcal{N}_d(0,I_d)) \eqlaw \mathcal{N}_d(m,A^TA)$. The L-moments of a multivariate Gaussian of mean $m$ and covariance $A^TA$ are :
\begin{align*}
\lambda_\alpha &=  \int_{[0;1]^d} \left[m + A\mathcal{N}_d(t_1,...,t_d) \right]L_\alpha (t_1,...,t_d) dt_1...dt_d\\
&= \mathds{1}_{\alpha=(1,...,1)}m + \mathds{1}_{\alpha\neq(1,...,1)}A\lambda_\alpha (\mathcal{N}_d(0,I_d))
\end{align*}
with the notation $\lambda_\alpha (\mathcal{N}_d(0,I_d))$ denoting the $\alpha$-th L-moments of the standard multivariate Gaussian, which is easy to compute since it is a random vector with independent components (see example \ref{ex_independent}).\\
In particular, the L-moment matrix of degree 2 : 
\begin{equation}
\Lambda_2 = \left(\lambda_{2,1...,1}\dots \lambda_{1,...,1,2}\right) = A \left(\begin{array}{ccc}\frac{1}{\sqrt\pi}&0&\dots\\0&\ddots&0\\ \dots&0 & \frac{1}{\sqrt\pi}\end{array} \right).
\end{equation}
The matrix of L-moments ratio of degree 2 is then 
\begin{equation}
\tau_2 = \left(\tau_{2,1,...,1}\dots \tau_{1,...,1,2}\right) = 
\left(\begin{array}{ccc} \frac{a_{11}}{\left(\sum_{i=1}^d a_{i1}^2\right)^{1/2}} & \dots& \frac{a_{1d}}{\left(\sum_{i=1}^d a_{id}^2\right)^{1/2}}\\
\vdots &\ddots &\vdots\\
\frac{a_{d1}}{\left(\sum_{i=1}^d a_{i1}^2\right)^{1/2}} & \dots& \frac{a_{dd}}{\left(\sum_{i=1}^d a_{id}^2\right)^{1/2}}
\end{array}\right) 
\end{equation}
with
\begin{equation*}
A= 
\left(\begin{array}{ccc} a_{11}& \dots& a_{1d}\\
\vdots &\ddots &\vdots\\
a_{d1}& \dots& a_{dd}
\end{array}\right) 
\end{equation*}

\end{example}

\begin{example}(Spherical and nearly-elliptical distributions)\\
We now present a generalization of the previous example close to the elliptical family. Let $u :\mathbb{R}\rightarrow \mathbb{R}$ be a derivable strictly convex function,$m\in\mathbb{R}$, $A$ be a positive matrix and define the potential : 
\begin{equation}
\varphi(x) = m.x + \frac{1}{2} u(x^TAx) \ \ \ \ \ \ \text{  for   } x\in\mathbb{R}^d.
\end{equation}
The associated transport is given by : 
\begin{equation}
T_0(x) = m+ u'(x^TAx)Ax.
\end{equation}
If the integral is well defined, the L-moments of this distribution are then 
\begin{equation}
\lambda_\alpha = \mathds{1}_{\alpha=(1,...,1)}m + A\int_{\mathbb{R}^d}  u'(x^TAx)xL_\alpha (\mathcal{N}(x)) d\mathcal{N}(x).
\end{equation}

If we take $A=I_d$ and write $u'(x) = \frac{v(x)}{x^{1/2}} $, then $T_0(X) = m + v(X^TX)\frac{X}{(X^TX)^{1/2}}$ where $X$ is a standard Gaussian random variable which is the characterization of a spherical distribution according to \cite{cambanis81}.

\begin{figure}
  \centering
   \subfloat{  \includegraphics[width=3in]{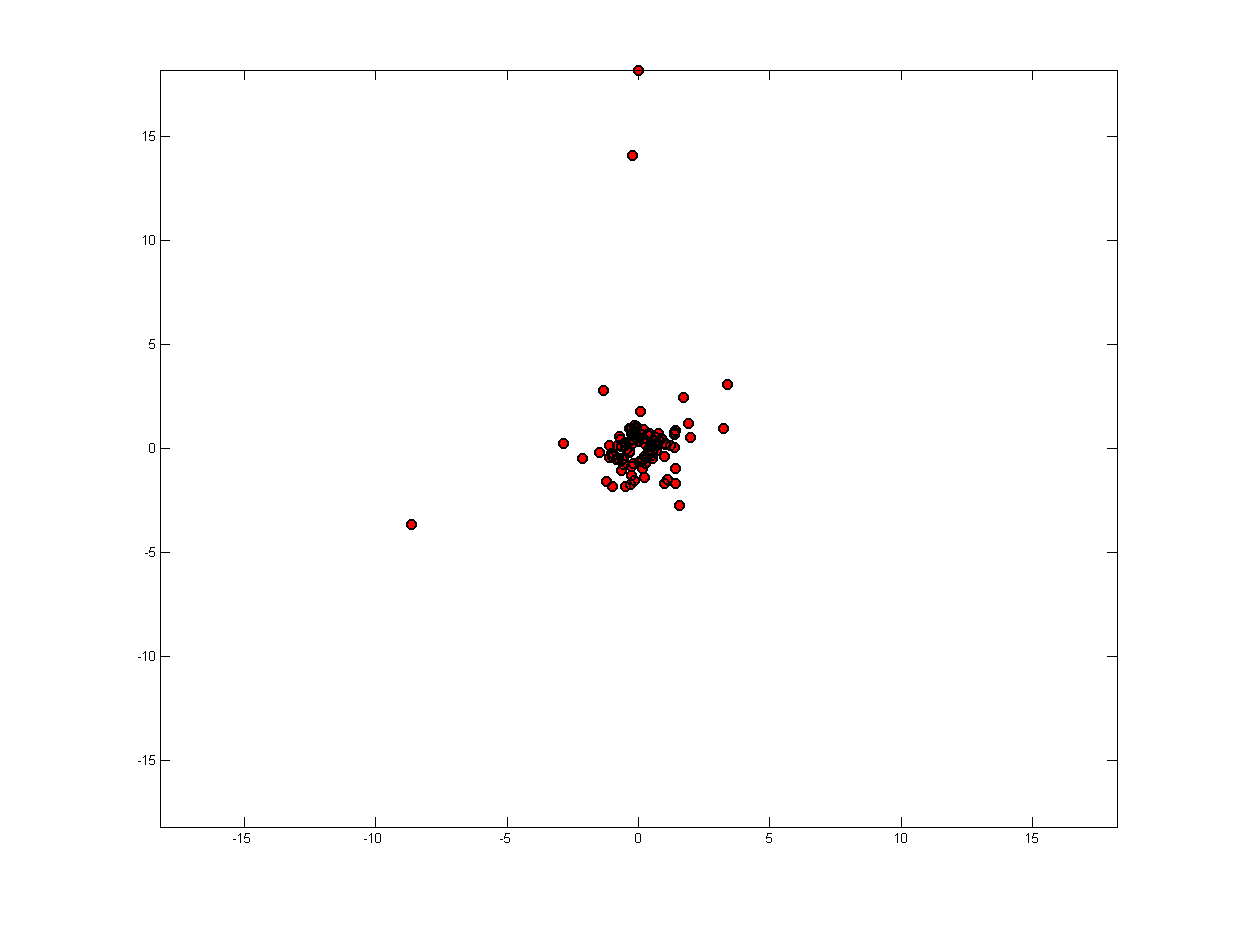}}
     \subfloat{ \includegraphics[width=3in]{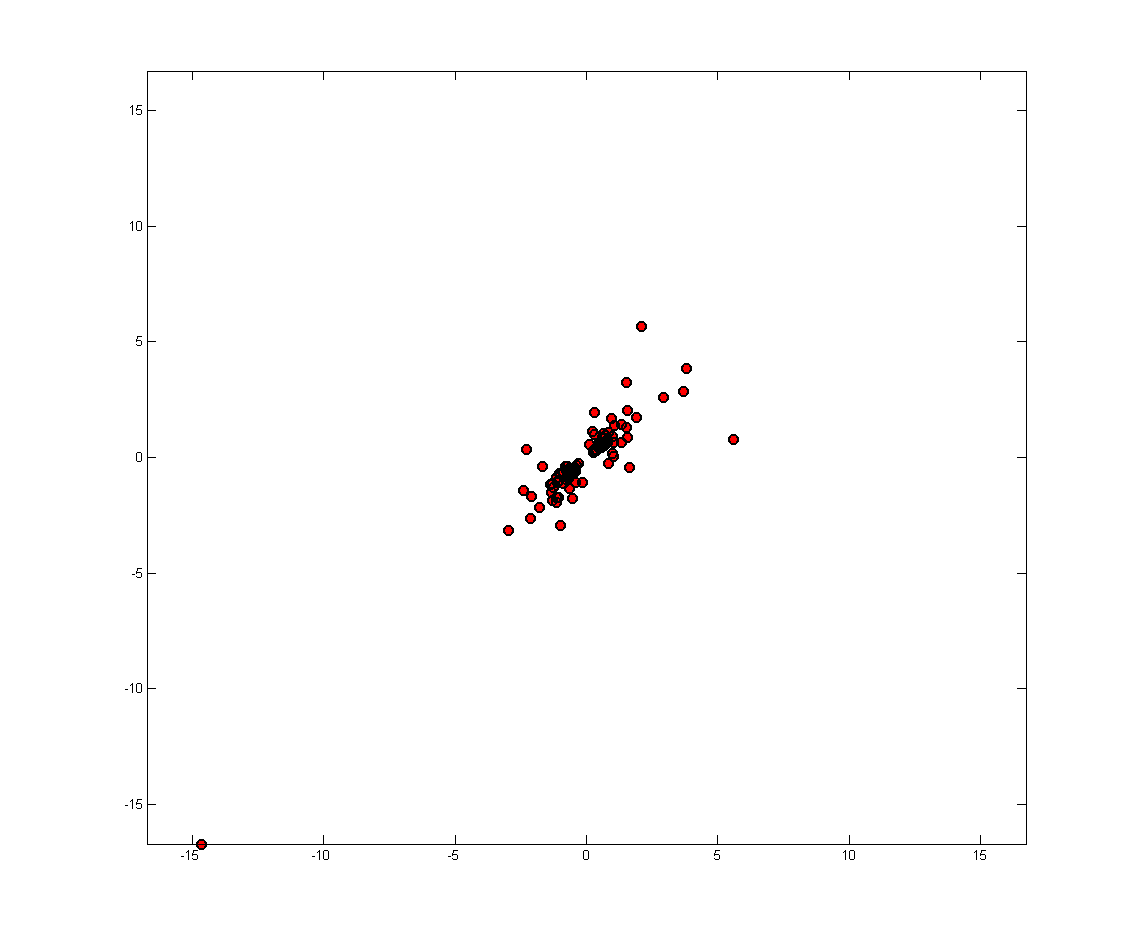}}
   \caption[Samples from the distribution induced by $T_0(x) = u'(x^TAx)Ax$ with $u(x) =-\log(x)$]{Samples from the distribution induced by $T_0(X) = u'(x^TAx)Ax$ with $u(x) =-\log(x)$ and $A=I_d$ (left) or $A=\left(\begin{array}{cc}1 & 0.8\\0.8&1\end{array}\right)$ (right)}
\end{figure}

\begin{figure}
  \centering
   \subfloat{  \includegraphics[width=3in]{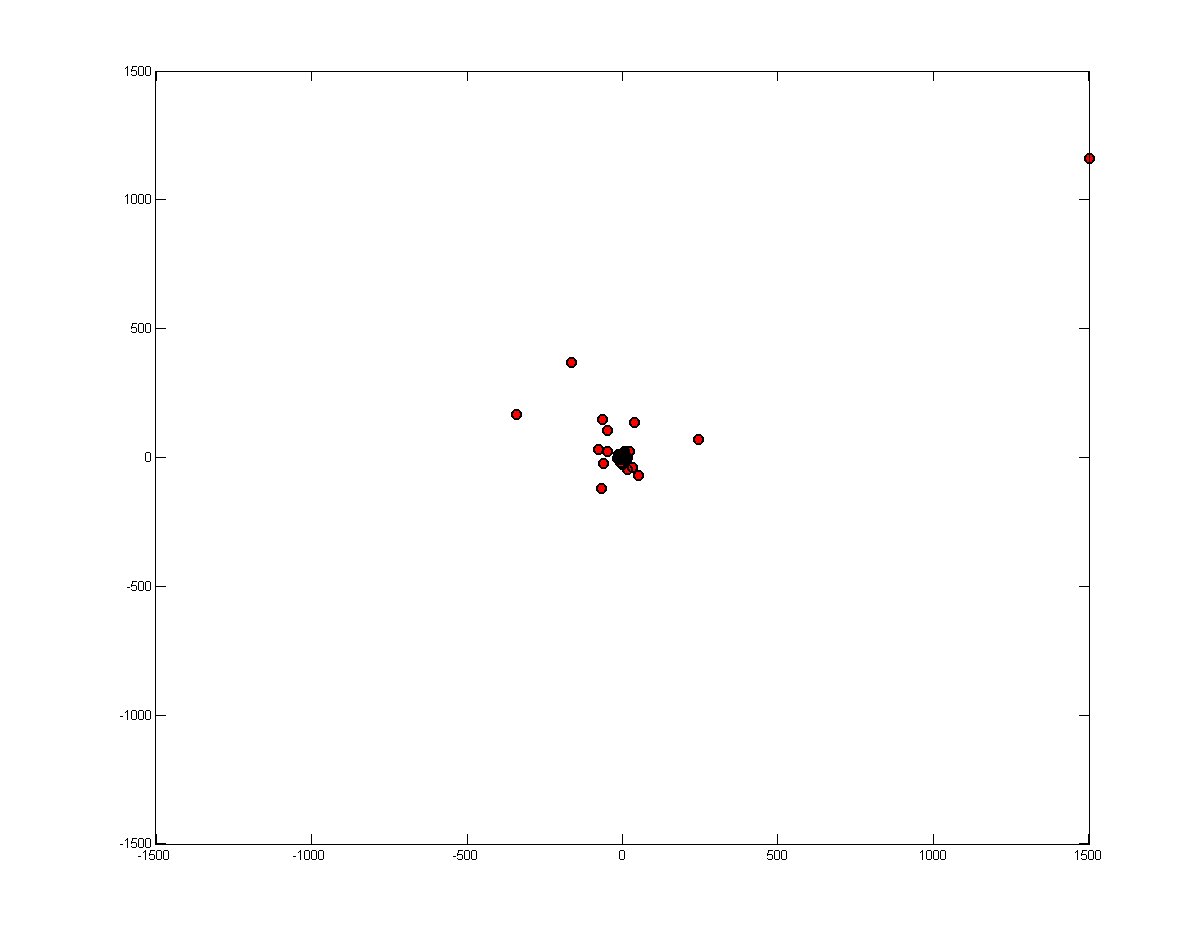}}
     \subfloat{ \includegraphics[width=3in]{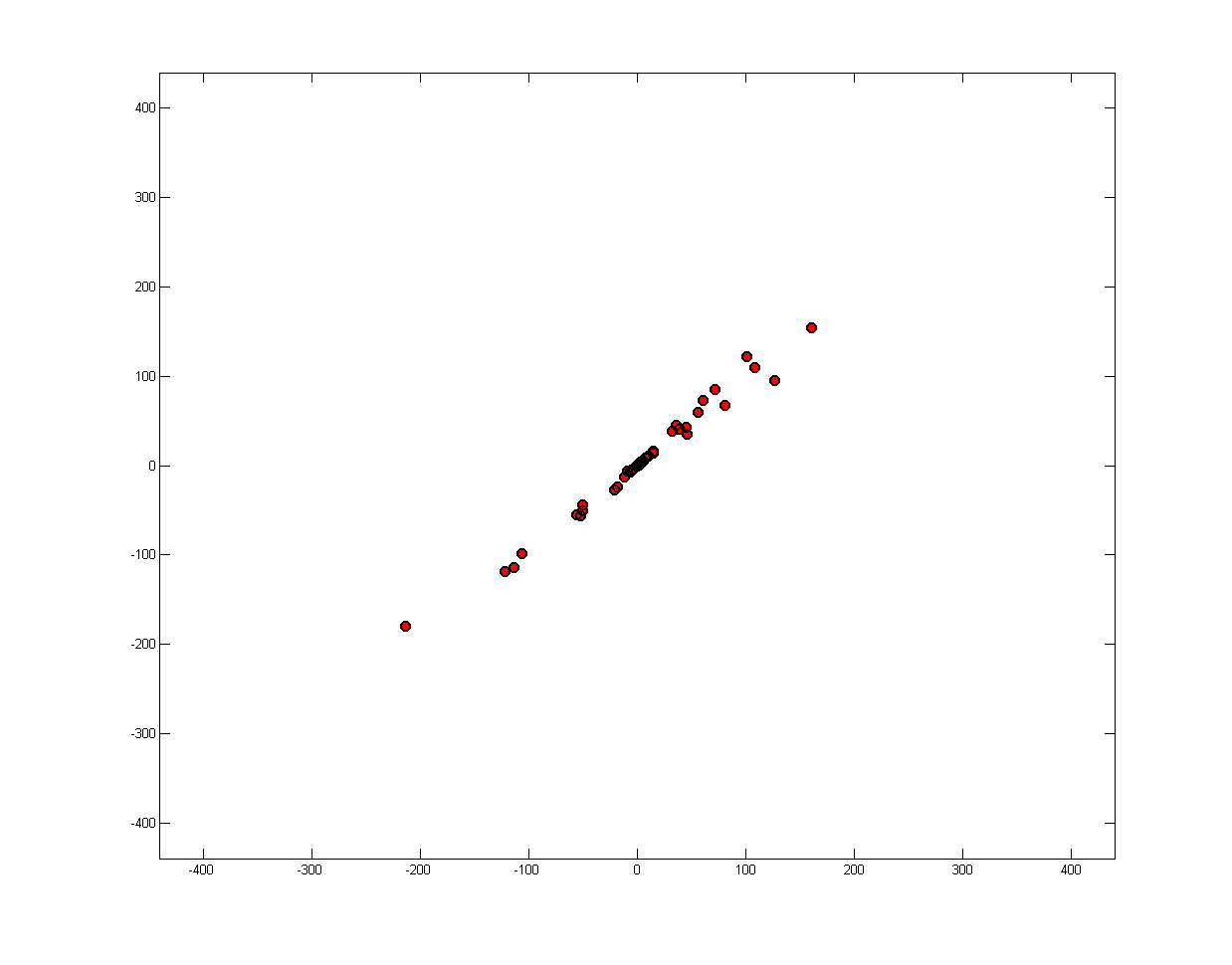}}
   \caption[Samples from the distribution induced by $T_0(x) = u'(x^TAx)Ax$ with $u(x) =\frac{1}{3}x^3$]{Samples from the distribution induced by $T_0(X) = u'(x^TAx)Ax$ with $u(x) =\frac{1}{3}x^3$ and $A=I_d$ (left) or $A=\left(\begin{array}{cc}1 & 0.8\\0.8&1\end{array}\right)$ (right)}
\end{figure}

\end{example}

\begin{example}(Linear combinations of independent variables)\\
\label{lciv1}
Let $(e_1,...,e_d)$ be an orthonormal basis of $\mathbb{R}^d$ and $(b_1,...,b_d)$ the canonical basis. We consider the potential defined by :
\begin{equation}
\varphi(x) = \sum_{i=1}^d \sigma_i \varphi_i(x^Te_i)
\end{equation}
with each function $\varphi_i$ derivable and convex and $\sigma_i>0$. Then
\begin{equation}
\nabla\varphi(x) = \sum_{i=1}^d \sigma_i \varphi_i'(x^Te_i)e_i 
\end{equation}
Then, if we denote by $P=\sum_{i=1}^d  e_ib_i^T$ and $D=\sum_{i=1}^d \sigma_ib_ib_i^T$, this potential generates the random vector 
\begin{equation}
Y\eqlaw P^TD \myvec{\varphi'_1(X^Te_1)}{\varphi'_d(X^Te_d)}.
\end{equation}
Let us note that $P$ is orthogonal i.e. $PP^T=P^TP=I_d$ and $D$ is diagonal.\\
As $e_1$,...,$e_d$ is an orthonormal family, $X^Te_1,...,X^Te_d$ are independent Gaussian random variables. Then if we write the increasing functions $\varphi'_i(x) = Q_i(\mathcal{N}_1(x))$ with $Q_i$ the quantile of a random variable $Z_i$, then
\begin{equation}
Y\eqlaw P^T \myvec{\sigma_1 Z_1}{\sigma_dZ_d}
\label{eq:lciv}
\end{equation}
with $Z_1$,...,$Z_d$ independent. The parameters $\sigma_i$ are meant to represent a scale parameter for each $Z_i$ but can be absorbed in the function $\varphi_i'$.\\
Figure \ref{fig:lciv} illustrates this model with for each $i$, $Z_i=\epsilon Z_i'$ where $\epsilon$ is a Rademacher random variable (i.e. discrete with probability $\frac{1}{2}$ on $-1$ and $1$) and $Z_i'$ is a Weibull random variable.\\

\begin{figure}
  \centering
   \subfloat [Symmetrized Weibull density for $Z_1$ and $Z_2$ (shape parameter 0.5 and scale parameter 1)]{  \includegraphics[scale=0.3]{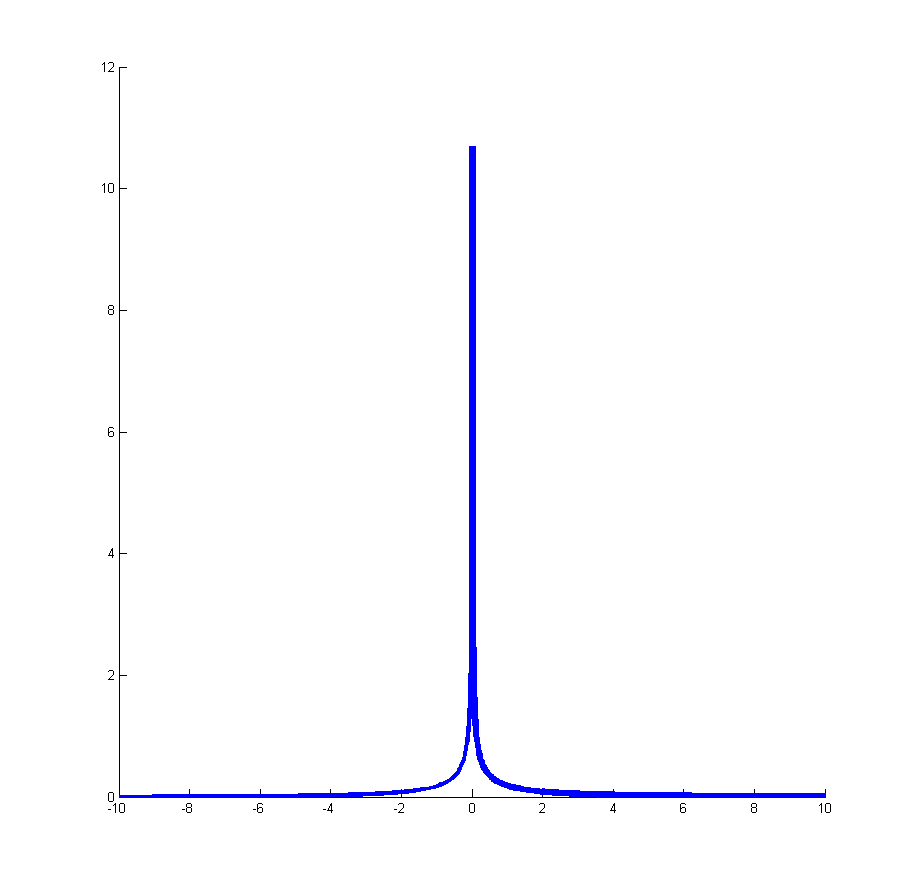}}
     \subfloat[Corresponding samples]{ \includegraphics[scale=0.3]{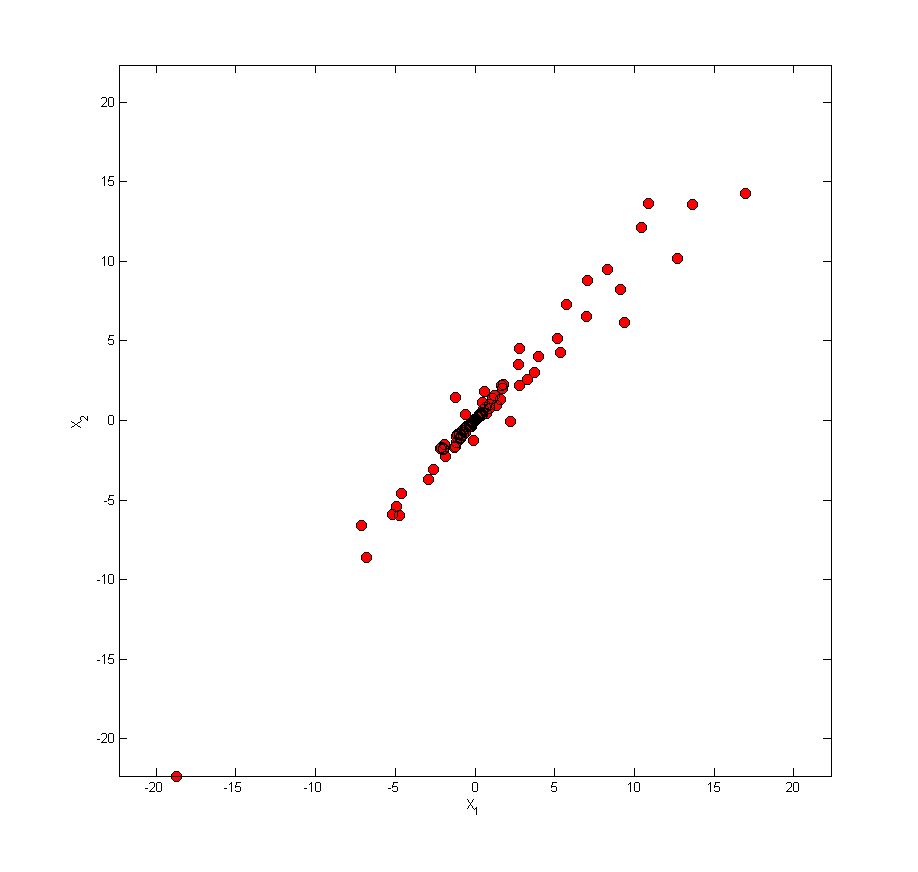}}\\
   \subfloat [Symmetrized Weibull density for $Z_1$ and $Z_2$ (shape parameter 2 and scale parameter 1)]{  \includegraphics[scale=0.3]{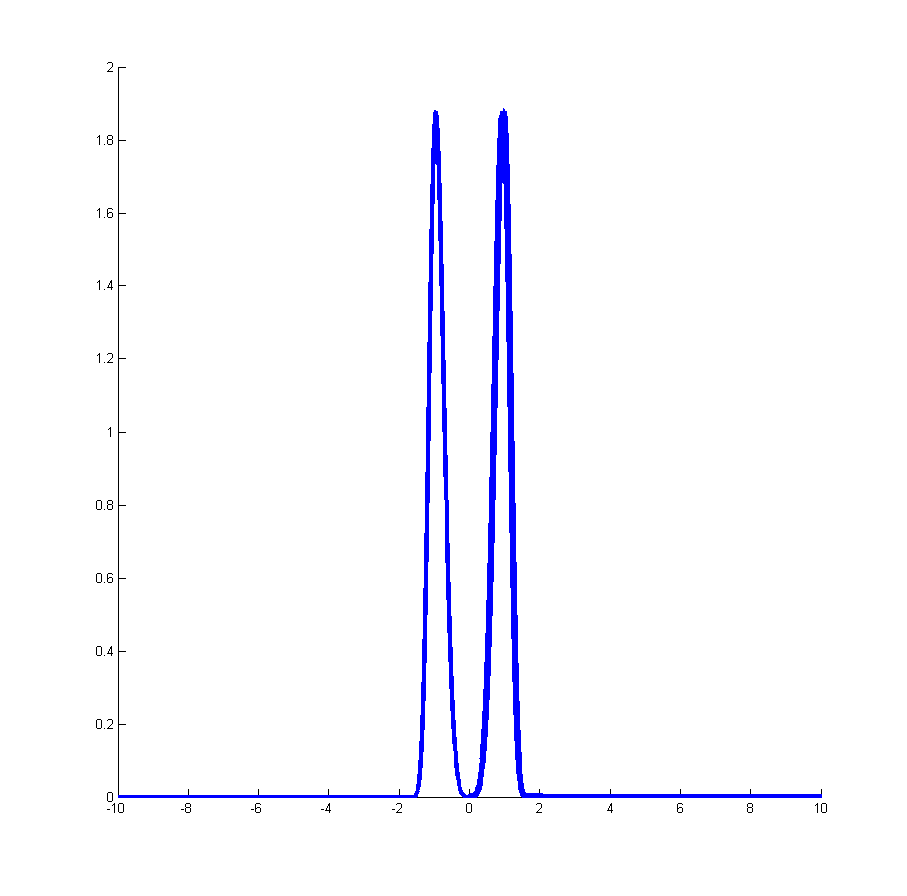}}
     \subfloat[Corresponding samples]{ \includegraphics[scale=0.3]{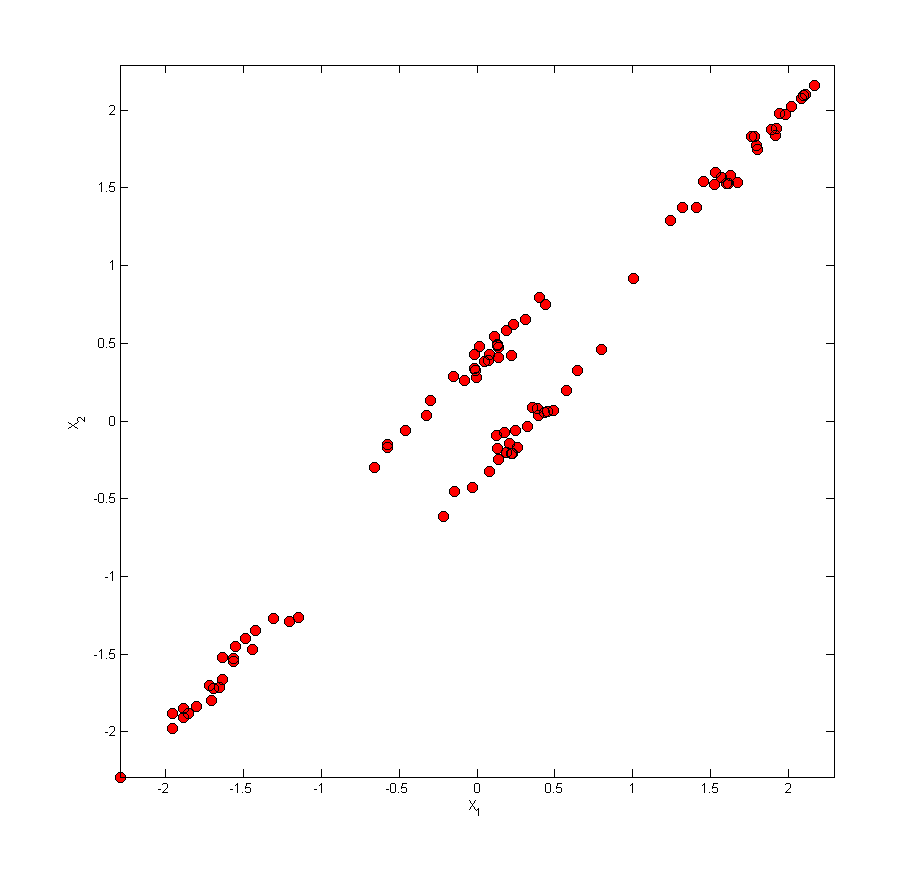}}\\
   \caption[Samples from the distribution given by Equation (\ref{eq:lciv})]{\label{fig:lciv}Samples from the distribution given by equation (\ref{eq:lciv}) with $P=\frac{1}{\sqrt{2}}\left(\begin{array}{cc}1 & 1\\1&-1\end{array}\right)$, $\sigma_1=1.8$ and $\sigma_2=0.2$ (right) for different parameters for the Weibull distribution}

\end{figure}

The L-moments of $Y$ are then for $\alpha\in\mathbb{N}^d_*$ :
\begin{equation*}
\lambda_{\alpha} = P^TD \myvec{\int_{\mathbb{R}^d}\varphi_1'(\langle x,e_1\rangle)L_{\alpha}(\mathcal{N}_d(x))d\mathcal{N}_d(x)}{\int_{\mathbb{R}^d}\varphi_d'(\langle x,e_d\rangle)L_{\alpha}(\mathcal{N}_d(x))d\mathcal{N}_d(x))}
\end{equation*}

\end{example}

\section{Rosenblatt transport and L-moments}

\subsection{General multivariate case}

In a paper dated of 1952 \cite{rosenblatt52}, Rosenblatt defined a transformation for the random variable $X=(X_1,...,X_d)$ with an absolutely continuous distribution. This transformation denoted by $T$ is now known as Rosenblatt transport (sometimes named Knothe's transport) and is explicitly given by the successive conditional distributions of $X_k | X_1=x_1,...,X_{k-1}=x_{k-1}$ :
\begin{equation}
\label{eq:rosenblatt}
T(x_1,...,x_d) = \left(\begin{array}{c}
F_{X_1}(x_1)\\ F_{X_2|X_1}(x_2|x_1),\\ \vdots\\ F_{X_d|X_1,...,X_{d-1}}(x_d|x_1,...,x_{d-1})
\end{array}\right)
\end{equation}
Rosenblatt showed that $T$ transports the random variable $X$ into the uniform law on $[0;1]^d$. However, $T$ is not uniquely defined because there are $d!$ transports $T$ corresponding to the $d!$ ways in which one can number the coordinates $X_1,...,X_d$.\\
In the following, we soften the absolute continuity assumption in order to transport the uniform measure on $[0;1]^d$ $\mu$ onto an arbitrary measure $\nu$. In that version, the Rosenblatt transport is based on the disintegration theorem (given without proof) which is a consequence of the Radon-Nikodym theorem (see for example \cite{ambrosio05}) : 
\begin{theorem}
Let $E_1$ and $E_2$ be two separable metric spaces equipped with their Borel $\sigma$-algebras, $B_{E_1}$ and $B_{E_2}$. Let $\gamma$ be a Borel probability measure on $E_1\times E_2$ and $\gamma_1 = \pi_{E_1}\gamma$ be its first marginal; then there exists a family of probability measures on $E_2$, $(\gamma_2^{x_1})_{x_1\in E_1}$ measurable in the sense that $x_1\mapsto \gamma_2^{x_1}(A_2)$ is $\mu$-measurable for every $A_2\in B_{E_2}$ and such that $\gamma =\gamma_1\otimes\gamma_2^{x_1}$ i.e. : 
\begin{equation}
\gamma(A_1\times A_2) = \int_{A_1}\gamma_2^{x_1}(A_2)d\gamma_1(x_1)
\end{equation}
for every $A_1\in B_{E_1}$ and $A_2\in B_{E_2}$.
\end{theorem}
We can sum up the previous theorem by stating the existence of measures $\gamma_2^{x_1}$ such that
\begin{equation}
\gamma = \gamma_1 \otimes \gamma_2^{x_1}.
\end{equation}
$\gamma_2^{x_1}$ correspond to the notion of conditional distribution of the second marginal of $\gamma$ knowing the first marginal is equal to $x_1$. The disintegration can be a way to define conditioning according to Chang and Pollard \cite{chang97}. If $\gamma$ is absolutely continuous and we have denoted its density by $p$, the disintegrated measures $\gamma_1$ and $\gamma_2^{x_1}$ have respective densities : 
\begin{equation*}
p_1(x_1) := \int p(x_1,x_2)dx_2 \text{\ \    and\ \    } p_2^{x_1}(x_2) := \frac{p(x_1,x_2)}{p_1(x_1)}.
\end{equation*}

The Rosenblatt transport refer to the concatenation of univariate transports of disintegrated measures from $\nu$. More precisely in the case of the quantile, we recall that $\nu$ is a probability measure defined on the Borelian of $\mathbb{R}^d$. Let denote $\nu_1$, $\nu_2^{x_1}$, ..., $\nu_d^{x_1,\dots,x_{d-1}}$ the disintegration of $\nu$ and $F_1$, $F_2^{x_1}$,..., $F_d^{x_1,...,x_{d-1}}$ the corresponding cdf. Then the Rosenblatt quantile is defined by
\begin{equation}
Q(t_1,t_2,...,t_d) := \left(\begin{array}{c}
Q_1(t_1)\\
Q_2(t_1,t_2)\\
\vdots\\
Q_d(t_1,...,t_d)
\end{array}\right)
=
 \left(\begin{array}{c}
F_1^{-1}(t_1)\\
(F_2^{Q_1(t_1)})^{-1}(t_2)\\
\vdots\\
(F_d^{Q_1(t_1),...,Q_{d-1}(t_1,...,t_{d-1}) })^{-1}(t_d)
\end{array}\right).
\label{eq:rosenblattquantile}
\end{equation}

This construction transports the uniform distribution on $[0;1]^d$ into the distribution of the random vector $X$ and can be defined even if the distribution of $X$ is not absolutely continuous.\\
\begin{prop}
If $U$ is the uniform distribution on $[0;1]^d$ and $Q$ is a Rosenblatt quantile, then $Q(U) \eqlaw X$.
\end{prop}
\begin{proof}
We will prove the statement in the bivariate case for $Q = Q_{12}$ with
\begin{equation*}
Q_{12}(t_1,t_2) = \left(\begin{array}{cc} Q_1(t_1)\\ Q_2(t_1,t_2) \end{array}\right)
\end{equation*}
The generalization to the multivariate case is very similar.\\
Let $a$ be a function $dF$-measurable, then : 
\begin{align*}
\int_{[0;1]^2} a(Q_{12}(u))du &= \int_{[0;1]^2} a(Q_1(u), (F_2^{Q_1(u)})^{-1}(v))dudv\\
&= \int_0^1 \int_{\mathbb{R}} a(Q_1(u), y) dF_2^{Q_1(u)}(y) du\\
&= \int_{\mathbb{R}} \int_{\mathbb{R}} a(x, y) dF_2^{x}(y)dF_1(x)\\
&= \int_{\mathbb{R}^2} a(x, y) dF(x,y).
\end{align*}
The second and third equalities hold because the successive quantiles are one-dimensional transports.
\end{proof}

\begin{remark}
Carlier et al. \cite{carlier09} showed that the Rosenblatt transport can be viewed as a limit of optimal transports. Indeed, they showed that if we consider the cost depending on a parameter $\theta$ : 
\begin{equation*}
c_{\theta}(x,y) = \frac{1}{2} \sum_{k=1}^d \theta^{k-1} |x_k-y_k|^2
\end{equation*}
then the mapping $T_{\theta}$ solving the optimal transport with such a cost converges in $L^2$ to the Rosenblatt transport given by equation (\ref{eq:rosenblatt}) as $\theta$ goes to 0. We see once again that the Rosenblatt transport depends on the numbering order of the coordinates $x_1,..,x_d$ because $c_{\theta}$ is not symmetric with respect to the coordinates of $x$ and $y$.

\end{remark}

\subsection{The case of bivariate L-moments of the form $\lambda_{1r}$ and $\lambda_{r1}$}

We now consider a bivariate vector $X=(X_1,X_2)$. The two possible Rosenblatt quantiles are given by the successive conditional quantiles
\begin{equation*}
Q_{12}(t_1,t_2) = \left(\begin{array}{cc} Q_{X_1}(t_1)\\ Q_{X_2|X_1=Q_{X_1}(t_1)}(t_2) \end{array}\right)
\end{equation*}
or
\begin{equation*}
Q_{21}(t_1,t_2) = \left(\begin{array}{cc} Q_{X_1|X_2=Q_{X_2}(t_2)}(t_1)\\ Q_{X_2}(t_2) \end{array}\right)
\end{equation*}
where $Q_{X_1}, Q_{X_2}$ are the marginal quantiles of $X_1$ and $X_2$ and $Q_{X_2|X_1}, Q_{X_1|X_2}$ are the conditional quantiles.\\

The associated L-moments are then  :
\begin{equation*}
\begin{array}{ll}
	&\lambda^{(12)}_{\alpha} = \int_{[0;1]^2} Q_{12}(t_1,t_2)L_{\alpha}(t_1,t_2) dt_1dt_2\\
\text{or }&	\lambda^{(21)}_{\alpha} = \int_{[0;1]^2} Q_{21}(t_1,t_2)L_{\alpha}(t_1,t_2) dt_1dt_2
\end{array}
\end{equation*}
Here, the multi-indices $\alpha$ are couples $(r,s)$ for $r,s\geq 1$.\\
If we consider the pairs $(r,1)$ and $(1,s)$ and denote by $\lambda_r(X_i)$ the r-th univariate L-moment of $X_i$, we can express the corresponding L-moments :
\begin{equation}
	\lambda^{(12)}_{r1} = \int_{[0;1]^2} Q_{12}(t_1,t_2)L_r(t_1)dt_1dt_2 = 
   \left (
   \begin{array}{cc}
     \lambda_r(X_1) \\
     \mathbb{E}[L_r\circ F_1(X_1)\mathbb{E}[X_2|X_1]] \\
   \end{array}
   \right )
\end{equation}
and
\begin{equation}
	\lambda^{(21)}_{1s} = \int_{[0;1]^2} Q_{21}(t_1,t_2)L_s(t_2) dt_1dt_2 =
   \left (
   \begin{array}{cc}
     \mathbb{E}[L_s\circ F_2(X_2)\mathbb{E}[X_1|X_2]] \\
     \lambda_s(X_2) \\
   \end{array}
   \right ).
\end{equation}

Serfling and Xiao \cite{xiao07} implicitly used this transformation for a bivariate vector to define multivariate L-moments. For a multivariate vector $X=(X_1,...,X_d)$, they considered each pair $(X_i,X_j)_{1\leq i,j\leq d}$ which avoids considering the $d!$ ways to build the Rosenblatt transport and allows a straightforward estimation through the concomitants of the samples as we will see in the next section.\\
They named r-th multivariate L-moments  as the $d\times d$ matrix $\Lambda_r$ 
\begin{equation*}
\Lambda_r =
\left(
   \begin{array}{cccc}
     \Lambda_{r,11} & \Lambda_{r,12} &\dots & \Lambda_{r,1d} \\
     \Lambda_{r,21} & \Lambda_{r,22} & \ddots & \vdots \\
     \vdots & \ddots & \ddots & \vdots \\
     \Lambda_{r,d1} & \dots &\dots & \Lambda_{r,dd} \\
   \end{array}
   \right)
\end{equation*}
defined so that each $2\times 2$ submatrix is the concatenation of the above $2\times 1$ vectors : 
\begin{equation}
\left(
   \begin{array}{cc}
     \Lambda_{r,ii} & \Lambda_{r,ij} \\
     \Lambda_{r,ji} &  \Lambda_{r,jj} \\
   \end{array}
   \right)
=
\left(
   \begin{array}{ll}
     \lambda^{(ij)}_{1r} & \lambda^{(ji)}_{r1}\\
   \end{array}
   \right)
\end{equation}

\begin{example}
Unfortunately, these matrices are not sufficient for a total determination of a multivariate distribution. Let us present a copula that is an example of this assertion.\\
Let $\theta\in[-1;1]$ and $C_\theta(u,v)= uv + \theta K_a(u)K_b(v)$ for $u,v\in[0;1]$ with $a,b\geq3$. $C$ is a copula because :
\begin{itemize}
\item $C(1,v) = v$ for all $v\in[0;1]$, $C(u,1) = u$ for all $u\in[0;1]$ and $C(u,0) = C(0,v) = 0$ for all $u,v\in[0;1]$
\item if $u_1\leq u_2$ and $v_1\leq v_2$ : 
\begin{align*}
 C(u_2,v_2)& - C(u_1,v_2) - C(u_2,v_1) + C(u_1,v_1)\\=& (u_2 - u_1)(v_2 - v_1) + \theta(K_a(u_2)-K_a(u_1))(K_b(v_2)-K_b(v_1))\\ \geq& (1-\theta)  (u_2 - u_1)(v_2 - v_1) \geq 0
\end{align*}
because for all $a\geq 1$, $K_a$ is 1-Lipschitzian.
\end{itemize}
Furthermore, if we consider the matrices defined by Serfling and Xiao : 
\begin{equation*}
\Lambda_{r,11} = \Lambda_{r,11} =  \lambda_r(U([0;1]))
\end{equation*}
and
\begin{equation*}
\Lambda_{r,12} =  \mathbb{E}[L_r\circ F_1(X_1)\mathbb{E}[X_2|X_1]] = \int_{[0;1]^2} vL_r(u)dC_{\theta}(u,v) = \mathds{1}_{r=1}\frac{1}{2}.
\end{equation*}
Similarly, 
\begin{equation*}
\Lambda_{r,21} = \int_{[0;1]^2} uL_r(v)dC_{\theta}(u,v) = \mathds{1}_{r=1}\frac{1}{2}.
\end{equation*}
Hence, the whole family of cdf's $(C_{\theta})_{\theta\in[-1;1]}$ admits the same matrices $\Lambda_r$.
\end{example}

\textbf{Property of $\lambda_{r1}^{(12)}$ and $\lambda_{1r}^{(21)}$}\\

We will present properties for  $\lambda_{1r}^{(21)}$ that can be easily extended to  $\lambda_{r1}^{(12)}$. Although these specific L-moments do not completely characterize any bivariate distribution, they share some desirable properties.

\begin{prop}
Let us recall that the L-moments ratios are defined by (see Definition \ref{def_ratio})
\begin{equation*}
\tau_{1r}^{(21)} := \frac{\lambda_{1r}^{(21)}}{\lambda(X_2)}\in \mathbb{R}^2
\end{equation*}
Then, we have for $k=1,2$
\begin{equation}
|\langle \tau_{12}^{(21)},b_k\rangle|\leq 1
\end{equation}
\end{prop}
where $(b_1,b_2)$ is the canonical basis of $\mathbb{R}^2$.
\begin{proof}
We apply Proposition \ref{prop_ratio} with $V\eqlaw X_2$.
\end{proof}

Let us suppose $X_{1},X_{2},...,X_{r}$ $r$ bivariate random samples. If we order the samples along the second coordinate i.e. $X_{2,(1:r)}\leq X_{2,(2:r)}\leq ...\leq X_{2,(r:r)}$, the remaining first coordinate $X_{1,(i:r)}$, paired with each $X_{2,(i:r)}$, is named the concomitant of $X_{2,(i:r)}$ (see Yang \cite{yang77} for a general study of concomitants). Furthermore, note
\begin{equation*}
X^{(21)}_{(i:r)} = \left(\begin{array}{c} X_{1,(i:r)}\\X_{2,(i:r)}\end{array}\right).
\end{equation*}
The superscript $(21)$ refers to the choice of $X_2$ as sorting coordinate. We can then have an analogue characterization of the multivariate L-moment as a linear combination of expectations of concomitants.

\begin{prop}
The r-th L-moment may be represented as 
\begin{equation}
\lambda_{1r}^{(21)} = \frac{1}{r}\sum_{j=0}^{r-1} (-1)^j \dbinom{j}{r-1}\mathbb{E}[X_{(r-j:r)}^{(21)}]
\end{equation}
\end{prop}
\begin{proof}
Let $i\leq r$. We have $\mathbb{E}[X_{1,(i:r)}^{(21)}]=r\mathbb{E}[X_{1}|X_{2}=X_{2,(i:r)}]$ i.e. by analogy with standard order statistics
\begin{equation*}
\mathbb{E}[X_{1,(i:r)}^{(21)}] = r\dbinom{i-1}{r-1}\int_{-\infty}^{+\infty}\int_{-\infty}^{+\infty}x_1(F_2(x_2))^{i-1}(1-F_2(x_2))^{r-i}dF(x_1,x_2)
\end{equation*}
We continue the analogy with the dimension 1 to reorganize the coefficients and conclude.
\end{proof}

This characterization allows us to use L-statistics and U-statistics representation especially in order to build unbiased estimators (see \cite{xiao07}).

\section{Estimation of L-moments}

Let $x_{1},...,x_{n}$ be independently drawn from a common random variable $X\in\mathbb{R}^d$ of measure $\nu$. We will note $\nu_n=\sum_{i=1}^n \delta_{x_{(i)}}$ the empirical measure. The estimation of multivariate L-moments is built from an estimation of the quantile function, say $Q_n$. This section considers the estimation of $Q$, leading to explicit formulas for $Q_n$. The L-moments are estimated by plug-in, through
\begin{equation}
\hat\lambda_{\alpha} = \int_{[0;1]^d} Q_n(t)L_{\alpha}(t)dt
\end{equation}
The simplest idea for the estimation of $Q$ is to build the transport between the continuous uniform distribution on $\Omega=[0;1]^d$ and the discrete measure $\nu_n$ which is possible for the considered transports.

\subsection{Estimation of the Rosenblatt transport}

The estimation of this transport is attractive due to its simplicity and its similarity with the univariate case.\\
We suppose that the sampling distribution is \textbf{absolutely continuous} with respect to the Lebesgue measure. Then for two random samples $X_{i}$ and $X_{j}$, $\mathbb{P}[X_{i}=X_{j}]=0$.\\
If we denote by $Q_n$ the empirical quantile built from the construction of the equation \ref{eq:rosenblattquantile} for $\nu=\nu_n$, then $Q_n:[0;1]^d \rightarrow  \{x_{1},...,x_{n}\}$ is defined with probability 1 for all $1\leq i\leq n$ by :
\begin{equation*}
Q_n(u_1,...,u_d) = x_{(i:n)}^{(1)}\ \ \  \text{   for any   } u_1\in\left[\frac{i-1}{n};\frac{i}{n}\right), u_2,...,u_d\in[0;1]
\end{equation*}
where $x_{(1:n)}^{(1)}, x_{(2:n)}^{(1)}, \dots, x_{(n:n)}^{(1)}$ denote the samples sorted by their first coordinate. Recall that we call the $(d-1)$ last components of $x_{(i:n)}^{(1)}$ the concomitants of its first component \cite{yang77}.
\begin{remark}
If the sampling distribution is discrete for example, the expression of the quantile will be more complicated since the law of $X_2|X_1=x_1$ is not reduced to a single point.
\end{remark}

Therefore, a natural version for the estimated L-moments associated to the Rosenblatt quantile could be :
\begin{equation}
\hat\lambda_{(i_1,...,i_d)} = \int_{[0;1]^d} Q_n(u) L_{(i_1,...,i_d)}(u) du = \sum_{i=1}^n w_i^{(i_1)}x_{(i:n)}^{(1)}
\end{equation}
where $(i_1,...,i_d)\in\mathbb{N}_*^d$ and 
\begin{equation*}
w_i^{(i_1)}=\int_{(i-1)/n}^{i/n} L_{i_1}(u_1)du_1\mathds{1}_{i_2=1}...\mathds{1}_{i_d=1}
\end{equation*}
 are the weights of the estimators of the $ i_1$-th univariate L-moments. Therefore, this estimator has an interest only for L-moments of the form $\lambda_{i_1,1,...,1}$. We will restrict ourselves to this case for L-moments associated to Rosenblatt quantiles. Serfling and Xiao \cite{xiao07} proposed a slightly different estimator which is the unbiased version of the above estimator :
\begin{equation}
\hat\lambda_{(i_1,1,...,1)}^{(u)} = \sum_{i=1}^n v_i^{(i_1)}x_{(i:n)}^{(1)}
\end{equation}
with  
\begin{equation*}
v_i^{(i_1)} = \sum_{j=0}^{\min(i-1,i_1-1)} (-1)^{i_1-1-j} \dbinom{j}{i_1-1}  \dbinom{j}{i_1-1+j}  \dbinom{j}{n-1}^{-1}  \dbinom{j}{i-1}.
\end{equation*}

Moreover, the consistency of both estimators holds for bivariate random vectors but in general fails to hold for vectors of dimension $d>2$ :
\begin{theorem}
If we define for all $u\in[0;1]^d$ :
\begin{equation}
Q^{(1)}(u) = \myvecs{Q_{X_1}(u_1)}{Q_{X_2|X_1=Q_{X_1}(u_1)}(u_2)}{Q_{X_d|X_1=Q_{X_1}(u_1)}(u_d)}
\end{equation}
then
\begin{equation}
\hat\lambda_{(i_1,1,...,1)} \ps \int_{[0;1]^d} Q^{(1)}(u)L_{(i_1,1,...,1)}(u)du = \lambda_{(i_11,...,1)}
\end{equation}
and
\begin{equation}
\hat\lambda^{(u)}_{(i_1,1,...,1)} \ps \int_{[0;1]^d} Q^{(1)}(u)L_{(i_1,1,...,1)}(u)du = \lambda_{(i_11,...,1)}.
\end{equation}
\end{theorem}
\begin{proof}
The convergence of the first coordinate of $\hat\lambda_{(i_1,...,i_d)}$ or $\hat\lambda^{(u)}_{(i_1,...,i_d)}$ directly comes from the univariate L-moments convergence results \cite{hosking90}. The $(d-1)$ remaining coordinate converge as an application of the theorem of convergence for the linear combinations of concomitants \cite{yang77}.
\end{proof}

\begin{remark}
An other idea for the estimation of the Rosenblatt L-moments is to consider the Rosenblatt construction of the quantile with a smoothed version of the empirical distribution. If this smoothed version (for example a kernel version) is absolutely continuous with respect to the Lebesgue measure, then consistency would hold.
\end{remark}

\subsection{Estimation of a monotone transport}
\label{section_estimation}
We will show here that the monotone transport from any absolutely continuous distribution onto a discrete one is the gradient of a piecewise linear function. We present the construction of the monotone transport of an absolutely continuous measure $\mu$  defined on $\mathbb{R}^d$ onto $\nu_n=\sum_{i=1}^n \delta_{x_{i}}$. Here, $\mu$ will typically be either the standard Gaussian measure on $\mathbb{R}^d$ or the uniform measure on $[0;1]^d$. We will denote by $\Omega$ the support of $\mu$. 

\subsubsection{Power diagrams}
Here, we briefly present power diagrams, a tool generalizing Voronoi diagrams and coming from computational geometry, which is useful for the  representation of the discrete optimal transport.
\begin{definition}
Let $x_1, ..., x_n\in\mathbb{R}^d$ and their associated weights $w_1,...,w_n\in\mathbb{R}$. The power diagram of $(x_1,w_1), ..., (x_n,w_n)$ is the subdivision of $\Omega$ into $n$ polyhedra given by : 
\begin{equation}
\Omega = \bigcup_{1\leq i\leq n} PD_i =\bigcup_{1\leq i\leq n}  \left\{u\in\Omega \text{ s.t. } \|u-x_i\|^2 + w_i \leq \|u-x_j\|^2 + w_j\ \ \forall j\ne i \right\}
\end{equation}
\end{definition}
\begin{remark}
If the weights are all zero and $x_1,...,x_n\in\Omega$, then the power diagram is the Voronoi diagram. 
\end{remark}
Convex piecewise linear functions are strongly related to power diagrams through their gradient. Indeed, let $\phi_h : \Omega \rightarrow \mathbb{R}$ be a piecewise linear function. Assume that $\phi_h$ is parametrized by $h=\myvec{h_1}{h_n}\in\mathbb{R}^n$. Define then $\phi_h$ explicitly through :
\begin{equation}
\label{eq:potential}
\text{for any  } u \in\Omega, \ \ \  \phi_h(u) = \max_{1\leq i \leq n} \left\{u.x_i + h_i\right\}.
\end{equation}
Let $(W_i(h))_{1\leq i \leq n}$ be the polyhedron partition of $\Omega$ defined by 
\begin{equation*}
W_i(h) = \{u\in\Omega \text{  s.t. } \nabla \phi_h(u) = x_i\}.
\end{equation*}
 This subdivision is often called the natural subdivision associated to the piecewise linear function $\phi_h$. Then, we have the following lemma :
\begin{lemma}
The power diagram associated to $(x_1,w_1), ..., (x_n,w_n)$ is the polyhedron partition $\cup_{1\leq i\leq n} W_i(h)$ if $h_i = -\frac{\|x_i\|^2+w_i}{2}$.
\label{power_lemma}
\end{lemma}
\begin{proof}
The proof is straightforward since $ \nabla \phi_h(u) =x_i$ iff $u.x_i+h_i\geq u.x_j+h_j$ for all $j$ which is equivalent to $\|u-x_i\|^2+2h_i-\|x_i\|^2 \geq \|u-x_j\|^2+2h_j-\|x_j\|^2$.
\end{proof}

\subsubsection{Discrete monotone transport}

We will present a variational approach initially proposed by Aurenhammer \cite{aurenhammer87} for the quadratic optimal transportation problem between a probability measure $\mu$ defined on $\Omega$ and the empirical distribution of a sample $x_1,...,x_n$ which is denoted by $\nu_n$.\\
Let $\phi_h : \Omega \rightarrow \mathbb{R}$ be the piecewise linear function defined by Equation (\ref{eq:potential}).\\

\begin{theorem}
Let us suppose that $x_1,...,x_n$ are distinct points of $\mathbb{R}^d$. Let $\Omega$ be a convex domain of $\mathbb{R}^d$ such that $vol(\Omega)>0$ and $\mu$ an absolutely continuous probability measure with finite expectation.\\
Then $\nabla \phi_h$ is piecewise constant and is a monotone transport of $\mu$ into $\nu_n$ with a particular $h=h^*$, unique up to a constant $(b,...,b)$, which is the minimizer of an energy function $E$ 
\begin{equation}
h^* = \arg\min_{h\in\mathbb{R}^n} E(h) =  \arg\min_{h\in\mathbb{R}^n} \int_{\Omega}  \phi_h(u) d\mu - \frac{1}{n}\sum_{i=1}^n h_i.
\end{equation}
Furthermore, $E$ is strictly convex on
\begin{equation*}
H_0^{(n)}=\left\{h\in\mathbb{R}^n \text{ s.t. } \text{    for any   } 1\leq i \leq n\text{ , } W_i(h)\ne \emptyset \text{\ \ and\ \ } \sum_{i=1}^n h_i=0\right\}.
\end{equation*}
\end{theorem}
\label{opt_discrete}
\begin{proof}
The proof of Theorem 1.2 of Gu et al. can be extended to the case of a measure $\mu$ defined on an arbitrary convex set $\Omega\subset \mathbb{R}^d$. However, we do not prove that $\nabla E$ is a local diffeomorphism.\\
The proof is delayed to the Appendix.
\end{proof}
\begin{remark}
The convexity of the domain $\Omega$ is needed in order to ensure that $H_0^{(n)}$ is non void.
\end{remark}

As exposed in the Appendix, the gradient is simply given by :
\begin{equation}
\nabla E(h) = \left(   \begin{array}{cc}     \int_{W_1(h)}d\mu(x)-\frac{1}{n} \\	\vdots\\   \int_{W_n(h)}d\mu(x)-   \frac{1}{n} \\   \end{array}   \right).
\end{equation}
Moreover, for the expression of the Hessian of $E$, let us define the intersection faces for $1\leq i,j\leq n$ :
\begin{equation*}
\left\{
\begin{array}{ll}
F_{ij} = W_i(h) \cap W_j(h) \cap \Omega & \text{   if the codimension of $F_{ij}$ is 1}  \\
F_{ij} = \emptyset & \text{   otherwise}  
\end{array}
\right.
\end{equation*}

Then, if $dA$ denote the area form on $F_{ij}$, the Hessian of $E$ is given by
\begin{equation*}
\left\{ \begin{array}{ll}
\frac{\partial^2 E}{\partial h_i \partial h_j} = - \frac{1}{\|x_i-x_j\|} \int_{F_{ij}} dA & \text{  if $i\ne j$}\\
\frac{\partial^2 E}{\partial h_i \partial h_i} = \sum_{1\leq j\leq n, j\ne i} \frac{1}{\|x_i-x_k\|} \int_{F_{ij}} dA 
\end{array}
\right..
\end{equation*}
We can perform the computation of the solution $h^*$ of the minimization problem by Newton's method.\\
If $\Omega=[0;1]^d$, in order to initialize this algorithm with $h^{(0)}(x_1,...,x_n)\in H_0^{(n)}$, we consider the vector corresponding to the translation/scaling of the classical Voronoi cells into $[0;1]^d$ i.e. :
\begin{equation*}
h^{(0)}(x_1,...,x_n) = \frac{1}{4m_n} \myvec {\frac{1}{n}\sum_{i=1}^n |x_i|^2 - |x_1|^2} {\frac{1}{n}\sum_{i=1}^n |x_i|^2 - |x_n|^2}
\end{equation*}
with $m_n$ the largest coordinate absolute value among the sample $x_1,..,x_n$.

\begin{prop}
If $\Omega=[0;1]^d$ and the $x_i$'s are distinct :
\begin{equation*}
h^{(0)}(x_1,...,x_n)\in H_0^{(n)}.
\end{equation*}
\label{init_opt}
\end{prop}

\begin{proof}
Let us define the hypercube englobing all the samples $X_n = [-m_n;m_n]\times \dots\times [-m_n;m_n]$.
Let $V_1,..., V_n$ the Voronoi cells intersected with $X_n$. So with probability 1 :
\begin{equation*}
V_i = \{x\in X_n \text{ s.t. } |x-x_i|\leq |x-x_j|\ \ \forall j\ne i \}\ne \emptyset
\end{equation*}
It is clear that if $h_V = \myvec{-\frac{1}{2}|x_1|^2}{-\frac{1}{2}|x_n|^2}$, we have
\begin{equation*}
V_i = \{y\in X_n \text{ s.t. } \nabla \phi_{h_V}(y) = x_i\}
\end{equation*}
Let us note $u_c = \myvec{1/2}{1/2}$. Then, if $h_{\Omega} = \frac{1}{2m_n}h_V - \myvec{\langle x_1,u_c\rangle}{\langle x_n,u_c\rangle}$ and $u\in\Omega=[0;1]^d$ :
\begin{align*}
\phi_{h_{\Omega}}(u) &= \max_{1\leq i \leq n} \left\{ \langle u,x_i\rangle +h_{\Omega,i} \right\} \\
&= \frac{1}{2m_n} \max_{1\leq i \leq n}    \left\{   2m_n \langle \left(u- u_c \right),x_i\rangle + 2m_n \left(h_{\Omega,i}+ \langle x_i, u_c\rangle\right) \right\}\\
&= \frac{1}{2m_n} \phi_{h_V} \left(2m_n \left(u-u_c\right) \right) 
\end{align*}
So for any $1\leq i \leq N,\ W_i(h_{\Omega}) = 2m_n\left(V_i- u_c\right) \ne \emptyset$.\\
We end this proof by taking as initialization vector $h^{(0)}(x_1,...,x_n)= h_{\Omega} - \frac{1}{n}\sum_{i=1}^n h_{\Omega,i}$.
\end{proof}

\begin{remark}
If $\Omega=\mathbb{R}^d$, this initialization is not an issue since it suffices to take the vector $h$ corresponding to the Voronoi cells.
\end{remark}

\begin{algorithm}
\caption{Computation of the discrete optimal transport via Newton's method}
\begin{algorithmic}
\State $\textbf{Aim}$ : To compute the discrete subdivision of $\Omega$ designing the optimal transport between $\nu_n$ and $\mu$
\State $\textbf{Input}$ : $h_0\in H_0^{(n)}$, a descent step $\gamma$, a tolerance $\eta$
\State $\textbf{while  } |\nabla E(h_t)| >\eta$
\State  \hspace{0.5cm}$h_{t+1} = h_t-\gamma (\nabla^2 E(h_t))^{-1}\nabla E(h_t)$
\State  \hspace{0.5cm}$t\leftarrow t+1$
\State $\textbf{end}$
\end{algorithmic}
\label{algo_optimal}
\end{algorithm}

In practice, the Hessian $\nabla^2 E$ is often hard to compute since it requires the calculation of the area of the facets of a power diagram. In our implementation, we prefer to use the simpler gradient descent in Algorithm \ref{algo_optimal} : 
\begin{equation*}
h_{t+1} = h_t -\gamma \nabla E(h_t).
\end{equation*}
Moreover, in order to compute the gradient of $E$ for an arbitrary measure $\mu$, we use a Monte-Carlo method.\\
However, since $E$ is strictly convex only in $H_0^{(n)}$, Algorithm \ref{algo_optimal} may not converge to $h^*$ especially when $n$ is large. An improvement of this algorithm that would perform a gradient descent on the set $H_0^{(n)}$ is left as perspective.

\begin{figure}
  \centering
   \subfloat[Voronoi cells of the sample]{  \includegraphics[width=5in]{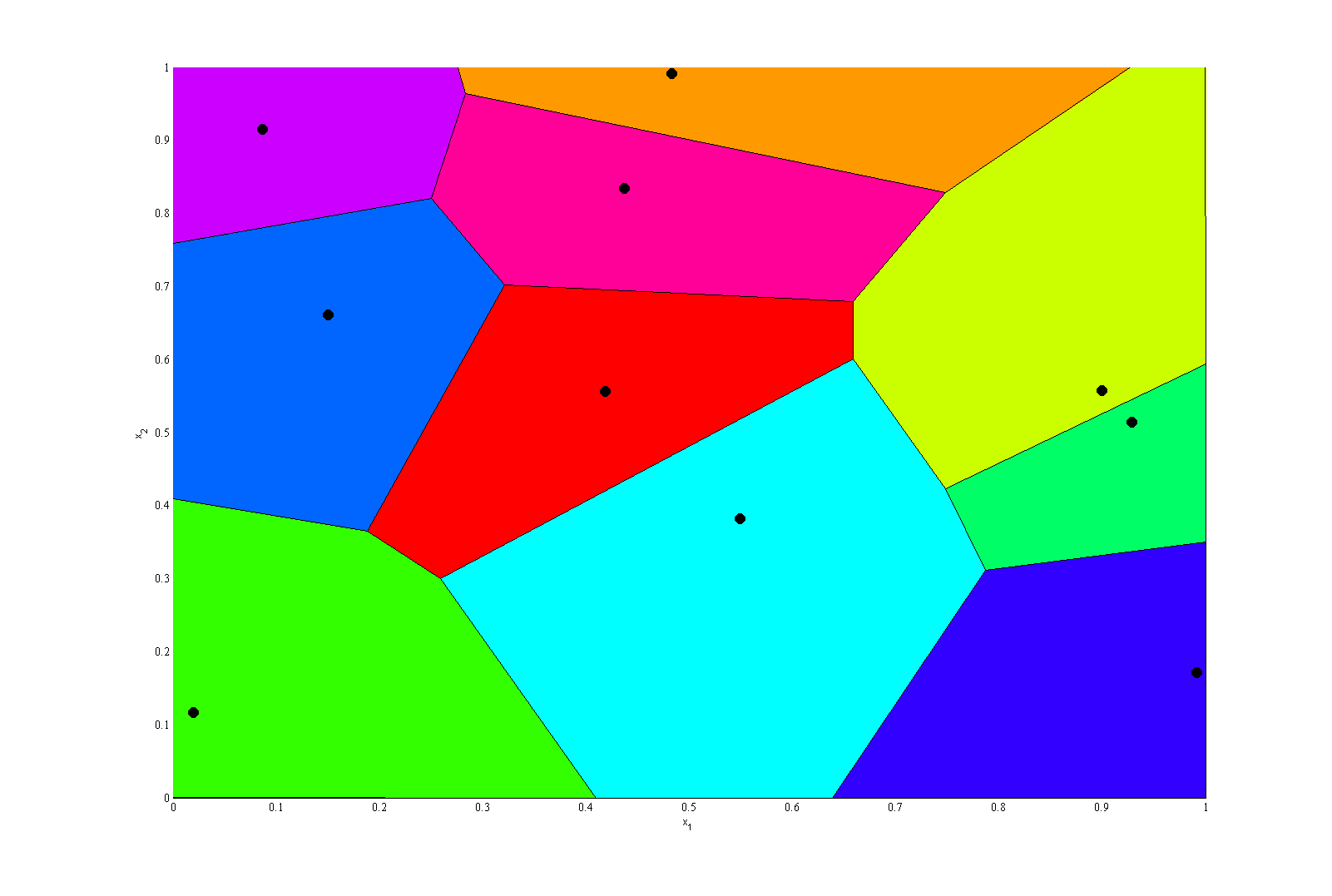}}\\
     \subfloat[Potential function of the optimal transport]{ \includegraphics[width=3in]{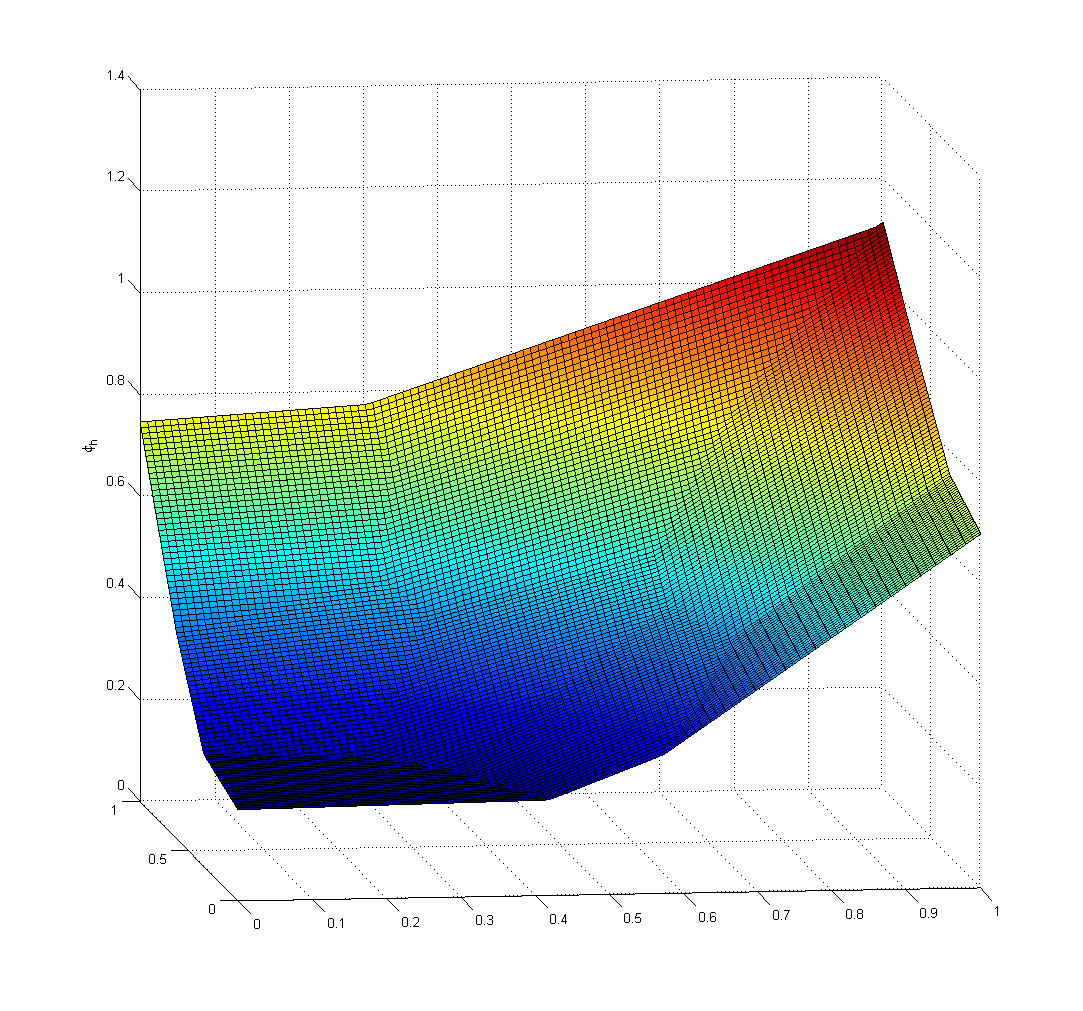}}
     \subfloat[Power diagram corresponding to the optimal transport (the transport maps each cell into one sample i.e. is piecewise linear)]{ \includegraphics[width=3in,height=2.5in]{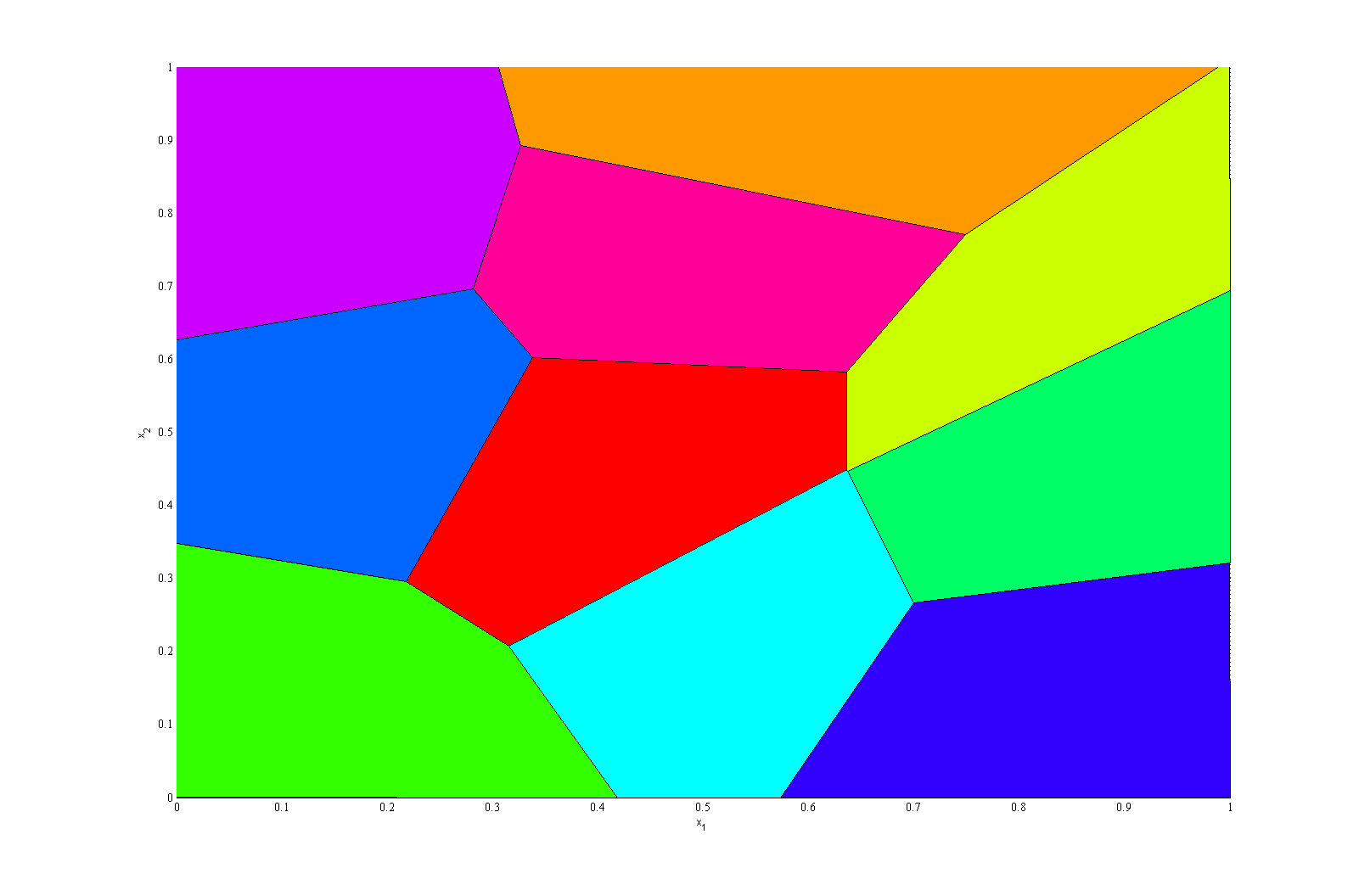}}
   \caption{Optimal transport of the discrete empirical distribution of a sample of size $10$ into the uniform distribution on $[0;1]^2$}
\end{figure}

\begin{figure}
  \centering
   \subfloat[Voronoi cells of the sample]{  \includegraphics[width=5in]{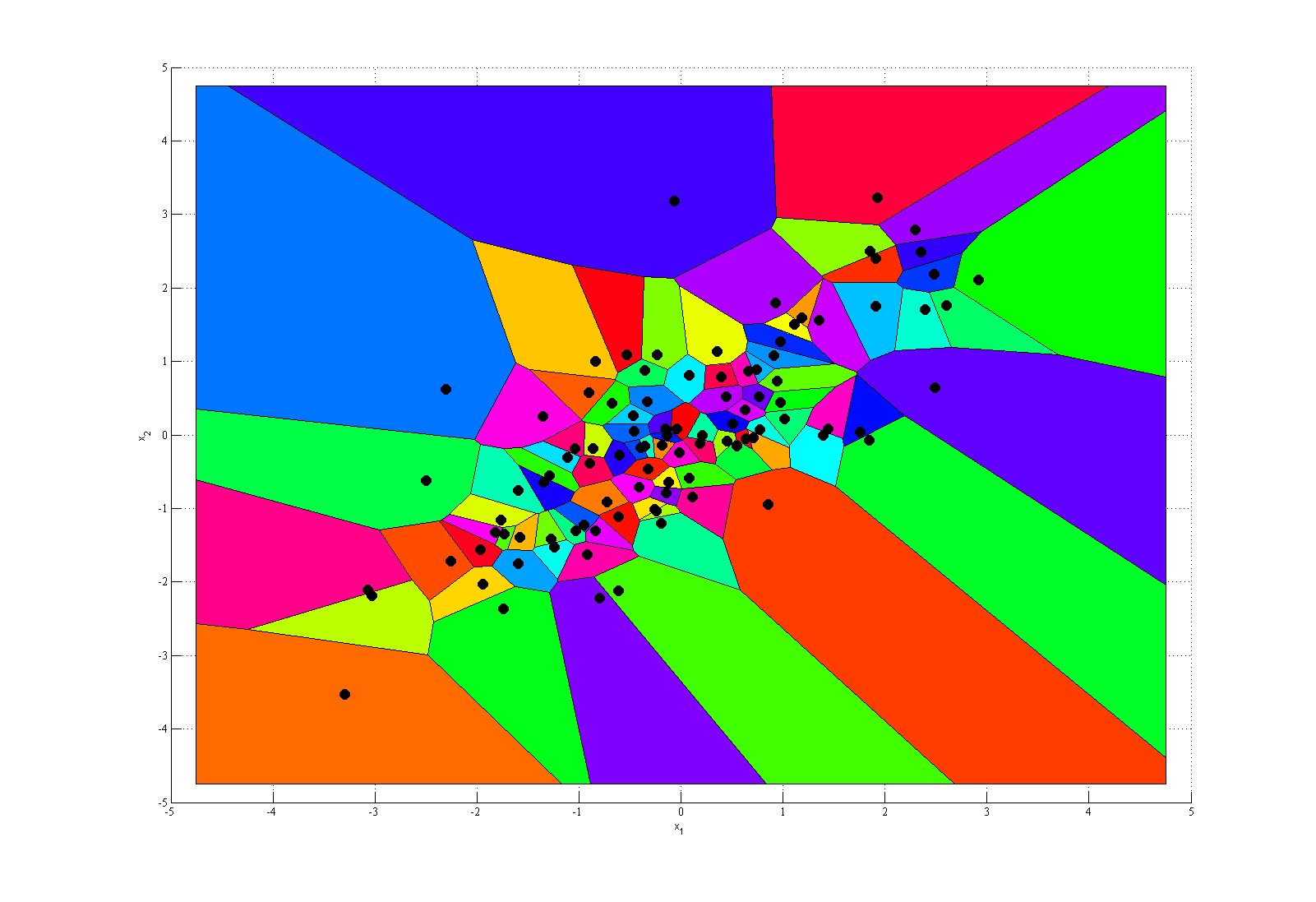}}\\
     \subfloat[Potential function of the optimal transport]{ \includegraphics[width=3in]{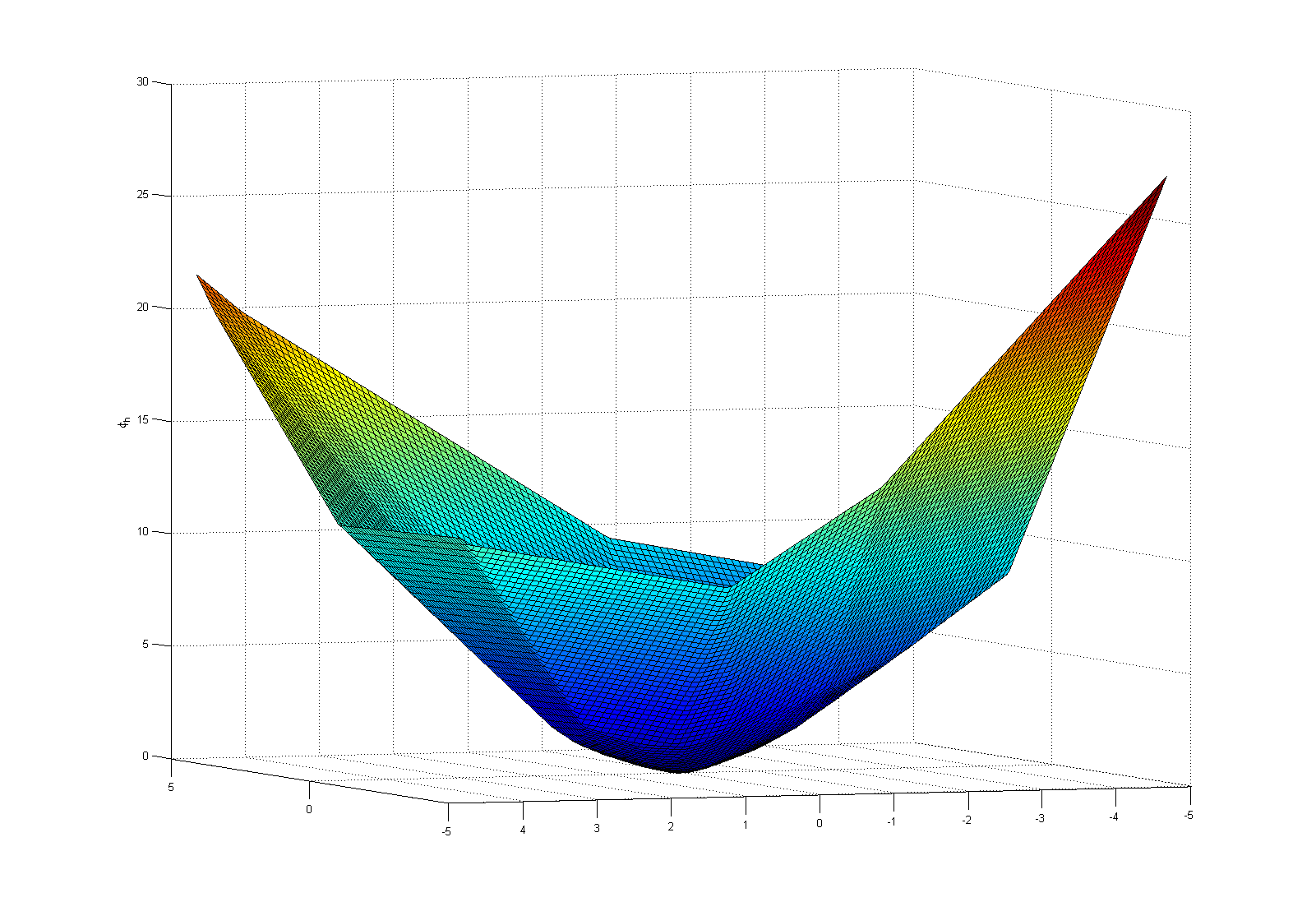}}
     \subfloat[Power diagram corresponding to the optimal transport (the transport maps each cell into one sample i.e. is piecewise linear)]{ \includegraphics[width=3in,height=2.5in]{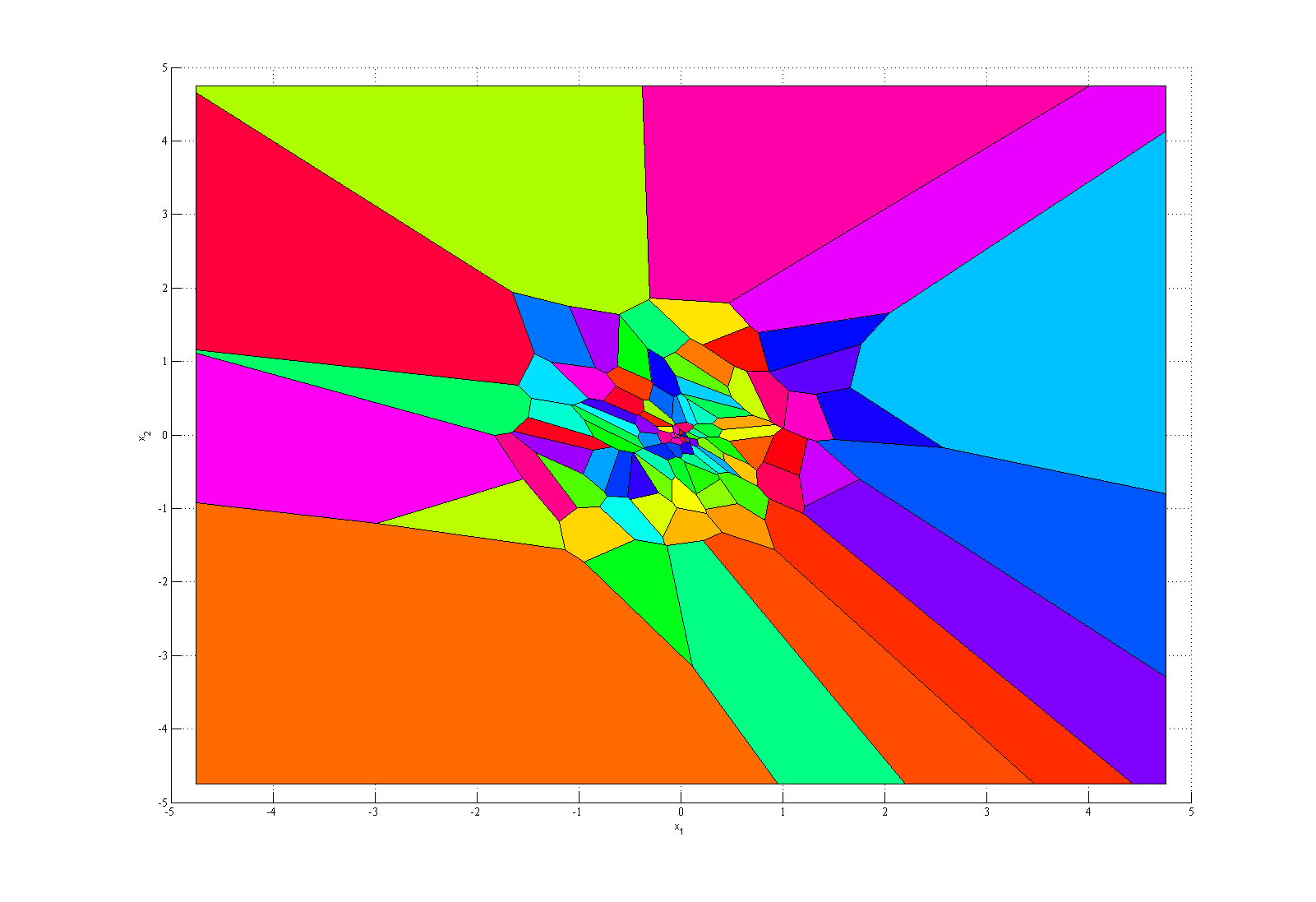}}
   \caption[Optimal transport for a sample of size 100 drawn from a Gaussian distribution into the standard Gaussian]{Discrete optimal transport for a sample of size $100$ drawn from a Gaussian distribution with covariance $\left(\begin{array}{cc}1 &0.8\\0.8&1\end{array}\right)$ into the standard Gaussian}
\end{figure}

\subsubsection{Explicit expression for 2 samples}
As an illustration, we can explicitly compute the monotone transport and some associated L-moments for 2 samples with a source distribution equal to the uniform on $[0;1]^d$ or to the standard normal $\mathcal{N}_d(0,I_d)$. Let $x_1, x_2\in\mathbb{R}^d$ two samples coming from the same distribution.\\
Let us first begin with the standard normal distribution as source measure. As the potential $\phi_h$ of the previous section is defined up to an additive constant, we consider for $h\in\mathbb{R}^d$ : 
\begin{equation*}
\phi_h(u) = \max(u.x_1,u.x_2+h)
\end{equation*}
$\nabla\phi_h$ is the discrete optimal transport if $W_i=\{y\in\mathbb{R^d} \text{ s.t. } \nabla\phi_h(y) = x_i\}$ for $i=1,2$ have a measure equal to 1/2 for the normal measure. By symmetry, we can assert that this property is attained for $h=0$. The transport is then for $y\in\mathbb{R^d}$ :
\begin{equation}
T_{\mathcal{N}}(y) = \nabla\phi_0(y) = \left\{
\begin{array}{l}
x_1 \text{      if $y.(x_1-x_2)\geq 0$}\\
x_2 \text{      if $y.(x_1-x_2)\leq 0$}
\end{array}
\right.
\end{equation}
The L-moments of degree 2 associated with this transport are then (for the sake of simplicity, we compute only the L-moments related to the first coordinate) : 
\begin{align*}
\lambda_{2,1,...,1}(x_1,x_2) &= \int_{\mathbb{R}^d} \nabla\phi_0(y) L_2(\mathcal{N}(y_1))d\mathcal{N}_d(y)\\
&= (x_1-x_2)\int_{y.(x_1-x_2)\geq 0} L_2(\mathcal{N}(y_1))d\mathcal{N}_d(y).
\end{align*}
Let us denote by $(e_1,..,e_d)$ the canonical basis of $\mathbb{R}^d$. Then, if $e_1$ and $x_1-x_2$ are collinear, then 
\begin{equation*}
\int_{y.(x_1-x_2)\geq 0} L_2(\mathcal{N}(y_1))d\mathcal{N}_d(y)=\pm1/4
\end{equation*}
depending on the sign of $y.(x_1-x_2)$. If it is not the case, we can build an orthonormal basis $(f_1=e_1,f_2,...,f_d)$ by completing the basis of the plan formed by $e_1$ and $x_1-x_2$. Let us denote by $U$ the rotation matrix transforming the canonical basis into the second one.\\
We also define $a_1= (x_1-x_2).f_1= (x_1-x_2).e_1$ and $a_2= (x_1-x_2).f_2$. Then, the integral becomes
\begin{align*}
\lambda_{2,1,...,1}(x_1,x_2) &= (x_1-x_2)\int_{a_1z_1+a_2z_2\geq 0} L_2(\mathcal{N}(z_1))d\mathcal{N}_d(z)\\
&= (x_1-x_2) \int_{\mathbb{R}} L_2(\mathcal{N}(z_1))\mathcal{N}(-\frac{a_1}{|a_2|}z_1)d\mathcal{N}(z_1)\\
&= \frac{x_1-x_2}{\pi} \arctan\left(\frac{a_1/|a_2|}{\sqrt{(a_1/|a_2|)^2 + 2}}\right)
\end{align*}
This last equality is obtained by deriving the function $t\mapsto \int L_2(\mathcal{N}(z_1))\mathcal{N}(tz_1)d\mathcal{N}(z_1)$ and is still valid for $a_2=0$ which corresponds to the case of collinearity of $e_1$ and $x_1-x_2$.\\

In the following, we will note the central point of the unit square $u_c=\myvec{1/2}{1/2}$. The same calculus can be performed when the source measure is uniform on the unit square. The transport is then by the same argument of symmetry for $u\in[0;1]^d$ : 
\begin{equation*}
T_{unif}(u) = \left\{
\begin{array}{l}
x_1 \text{      if $\left(u-u_c\right).(x_1-x_2)\geq 0$}\\
x_2 \text{      if $\left(u-u_c\right).(x_1-x_2)\leq 0$}
\end{array}
\right..
\end{equation*}
Performing the same kind of change of coordinate with a translation of $u_c$, the L-moments of order $(r,1,...,1)$  with $r\geq 2$ are :
\begin{align*}
\lambda_{r,1,...,1}(x_1,x_2) &= (x_1-x_2) \int_{\left(u-u_c\right).(x_1-x_2)\geq 0} L_r(u_1)du\\
&= (x_1-x_2)\int_{a_1v_1+a_2v_2\geq 0; |v_1|,|v_2|\leq 1/2} L_r(v_1+1/2)dv_1dv_2\\
&= \left\{ \begin{array}{ll}
\text{sgn}(a_1)K_r(\frac{1}{2}) & \text{     if $a_2=0$} \\
\frac{a_1}{6|a_2|} & \text{     if $\frac{|a_2|}{|a_1|}\geq 1$} \\
\frac{a_2}{2|a_2|}(1+\frac{a_1}{|a_1|}) \left[K_r(\frac{1}{2}(1+\left|\frac{a_2}{a_1}\right|)) - K_r(\frac{1}{2}(1-\left|\frac{a_2}{a_1}\right|)) \right] \\
- \frac{a_1}{|a_2|}\left[J_r(\frac{1}{2}(1+\left|\frac{a_2}{a_1}\right|)) - J_r(\frac{1}{2}(1-\left|\frac{a_2}{a_1}\right|)) \right]& \text{ otherwise} \\
\end{array}\right.
\end{align*}
where $K_r$ and $J_r$ are successive primitive functions of $L_r$.

\subsubsection{Consistency of the optimal transport estimator}

Let $X_1,..., X_n$ be $n$ independent copies of a vector $X$ in $\mathbb{R}^d$ with distribution $\nu$. Let $\mu$ denote a reference measure on a convex set $\Omega\subset\mathbb{R}^d$ so that $\mu$ gives no mass to small sets. Let $\nu_n$ be the empirical measure pertaining to the sample.\\
We define two transports, say $T$ and $T_n$ expressed as the gradient of two convex functions, say $\varphi$ and $\varphi_n$, so that $T=\nabla\varphi$ and $T_n=\nabla\varphi_n$.\\
$T$ and $T_n$ respectively transport $\mu$ onto $\nu$ and $\nu_n$. We do not assume that the hypothesis in Theorem \ref{gangbo} holds for $\mu$. Hence neither $T$ nor $T_n$ can be defined as an optimal transport for a quadratic cost; $T$ and $T_n$ are merely monotone transports.\\
This section is devoted to the statement of the convergence of $T_n$ to $T$.

\begin{definition}
A set $S\subset\mathbb{R}^d\times \mathbb{R}^d$ is said to be cyclically monotone if for any finite number of points $(x_i,y_i)\in S$, i=1...n
\begin{equation}
\langle y_1,x_2-x_1\rangle + \langle y_2,x_3-x_2\rangle + \dots \langle y_n,x_1-x_n\rangle \leq 0
\end{equation}
By extension, we say that a function $f$ is cyclically monotone if all subsets of the form 
\begin{equation*}
S = \left\{(x_1,f(x_1)),...,(x_n,f(x_n))\right\}
\end{equation*}
are cyclically monotone.
\end{definition}

Before stating consistency results, let us first begin with a lemma. 
\begin{lemma}
Let K be the space defined by 
\begin{equation}
K = \{\nabla \varphi\in L_1(\Omega,\mathbb{R}^d,\mu) \text{ , } \varphi \text{ convex $\mu$-a.e.}\}
\end{equation}
Then K is a Hilbert space for the norm :
\begin{equation}
\|\nabla\varphi\|_1 = \int_{\Omega} \|\nabla\varphi(x)\| d\mu(x)
\end{equation}
\end{lemma}

\begin{proof}
As $L_1(\Omega,\mathbb{R}^d,\mu)$ is a Hilbert space, it is sufficient to prove that $K$ is closed in $L_1(\Omega,\mathbb{R}^d,\mu)$. Let $\nabla\varphi_n$ be a sequence in $K$ convergent to $T\in L_1(\Omega,\mathbb{R}^d,\mu)$. $\nabla\varphi_n$ is cyclically monotone i.e. for all $m\in\mathbb{N}$ and $x_0, x_1, ..., x_m\in\Omega$ :
\begin{equation*}
(x_1-x_0).\nabla\varphi_n(x_0) +(x_2-x_1).\nabla\varphi_n(x_1) + \dots + (x_0-x_m).\nabla\varphi_n(x_m) \leq 0
\end{equation*}
Let $n\rightarrow\infty$ then :
\begin{equation*}
(x_1-x_0).T(x_0) +(x_2-x_1).T(x_1) + \dots + (x_0-x_m).T(x_m) \leq 0
\end{equation*}
i.e. $T$ is cyclically monotone. Furthermore, Theorem 24.8 of Rockafellar asserts that there exists a convex potential $\varphi$ whose subgradient is cyclically monotone \cite{rockafellar70}. As $\mu$ gives no mass to small sets, $\varphi$ is $\mu$-almost everywhere differentiable (see \cite{anderson52}) and $\nabla\varphi=T\in K$.
\end{proof}

\begin{lemma}(Lemma 9 McCann \cite{mccann95})\\
Let a sequence of probability measure on $\mathbb{R}^d\times \mathbb{R}^d$, denoted by $\gamma_n$, converge to $\gamma$ in the sense that for any test function $h\in C^{\infty}(\mathbb{R}^d\times \mathbb{R}^d)$ with compact support
\begin{equation*}
\int h(x,y)d\gamma_n(x,y) \rightarrow \int h(x,y)d\gamma_(x,y).
\end{equation*}
Let us call the marginals of $\gamma$ the respective measures $\mu$ and $\nu$ defined on $\mathbb{R}^d$ such that for any Borel set $M$ of $\mathbb{R}^d$
\begin{eqnarray*}
\mu(M)&=& \gamma(M\times \mathbb{R}^d)\\
\nu(M) &=& \gamma(\mathbb{R}^d \times M)
\end{eqnarray*}
Then
\begin{itemize}
\item $\gamma$ has a cyclically monotone support if $\gamma_n$ does for each n
\item if the marginals of $\gamma_n$, denoted by $\mu_n$ and $\nu_n$ converge in the sense given above to $\mu$ and $\nu$, then $\mu$ and $\nu$ are the respective marginals of $\gamma$
\end{itemize}
\label{lemma_mccann}
\end{lemma}
\begin{proof}
It is an application of McCann's  Lemma 9 \cite{mccann95}
\end{proof}

\begin{theorem}
\label{mccann2}
If $\nu$ satisfies $\int \|x\|d\nu(x)<+\infty$, let $T$ and $T_n$ be the monotone transports (i.e. gradients of convex function) of $\mu$ into $\nu$ and $\nu_n$. Then :
\begin{equation}
\|T-T_n\|_1=\int_{\Omega} \|T(x)-T_n(x)\|d\mu(x) \ps 0.
\end{equation}
\end{theorem}
\begin{proof}
$T$ is a gradient of a convex potential. We will consider the space 
\begin{equation*}
K = \{\nabla \varphi\in L_1(\Omega,\mathbb{R}^d,\mu) \text{ , } \varphi \text{ convex $\mu$-a.e.}\}.
\end{equation*}
$T_n$ and $T$ respectively transport $\mu$ into $\nu_n$ and $\nu$. By the strong law of large numbers :
\begin{equation*}
\int_{\Omega} \|T_n(x)\|d\mu(x) = \int_{\mathbb{R}^d} \|y\| d\nu_n(y) \ps \int_{\mathbb{R}^d} \|y\| d\nu(y).
\end{equation*}
Let $\omega$ be a realization such that :
\begin{equation*}
\int_{\Omega} \|T_n(\omega,x)\|d\mu(x) = \int_{\mathbb{R}^d} \|y\| d\nu_n(\omega,y) \rightarrow \int_{\mathbb{R}^d} \|y\| d\nu(\omega,y).
\end{equation*}
In the following, we will omit $\omega$ for the sake of simplicity of the notations.\\
We deduce from the convergence result of $\|T_n\|_1$ that $T_n$ is bounded for $n$ large enough. Hence, there exists $\nabla\psi \in K$ such that $T_m=\nabla\varphi_m \rightarrow \nabla\psi$ in $K$ for a subsequence $\{m\}$.\\
If we set $d\gamma_m(x,y) =\delta(y-\nabla\varphi_m(x))d\mu(x)$ and $d\gamma(x,y) =\delta(y-\nabla\psi(x))d\mu(x)$. Then for any function $f$ with compact support,  
\begin{equation*}
\int f(x,y) d\gamma_m(x,y)(x) \rightarrow \int f(x,y) d\gamma(x,y)(x).
\end{equation*}
By the above Lemma \ref{lemma_mccann}, we have that $\gamma$ have $\mu$ and $\nu$ as marginals, i.e. $\nabla\psi$ maps $\mu$ into $\nu$. By the uniqueness of the gradient of the convex transport, $\nabla\psi=\nabla\varphi=T$ is the unique limit point of the sequence $T_n$ in the Hilbert $K$.
\end{proof}

Let $T$ and $T_n$ be the transport of a reference measure $\mu_0$ onto $\nu$ and $\nu_n$ and $Q_0$ the transport of the uniform measure on $[0;1]^d$ onto this reference measure. Let us recall that we defined the quantiles of $\nu$ and $\nu_n$ by $Q=T\circ Q_0$ and $Q_n =T_n \circ Q_0$.

\begin{theorem}
Let $\nu$ satisfy $\int \|x\|d\nu(x)<+\infty$. Then, we have for $\alpha\in\mathbb{N}_*^d$.
\begin{equation}
\hat\lambda_{\alpha} = \int_{\Omega} Q_n(u)L_{\alpha}(u)du \ps \lambda_{\alpha} = \int_{\Omega} Q(u)L_{\alpha}(u)du
\end{equation}
\label{optimal_cons}
\end{theorem}
\begin{proof}
By using Theorem \ref{mccann2}, we have :
\begin{align*}
\|\hat\lambda_{\alpha} - \lambda_{\alpha}\| &= \left\| \int_{[0;1]^d} T_n(Q_0(t))L_{\alpha}(t) dt - \int_{[0;1]^d} T(Q_0(t))L_{\alpha}(t)dt \right\|\\
&\leq \left( \int_{[0;1]^d} \|T_n(Q_0(t))- T(Q_0(t))\|dt \right) \sup_{t\in[0;1]^d} L_{\alpha}(t)\\
&\leq  \int_{\mathbb{R}^d} \|T_n(x)- T(x)\|d\mu_0(x)\ps 0.
\end{align*}
\end{proof}

\begin{remark}
The L-moment estimator presented above has a L-statistic representation. For $\alpha\in\mathbb{N}_*^d$
\begin{equation*}
\hat\lambda_{\alpha} = \int_{[0;1]^d} T_n(Q_0(t))L_{\alpha}(t)dt = \sum_{i=1}^n \left(\int_{Q_0^{-1}(W_i(T_n))} L_{\alpha(t)}dt\right)x_{(i)}
\end{equation*}
where
\begin{equation*}
W_i(T_n) = \left\{x\in\mathbb{R}^d\text{   s.t.   } T_n(x)=x_i\right\}
\end{equation*}
\end{remark}

\section{Some extensions}
\subsection{Trimming}
\subsubsection{Semi-robust univariate trimmed L-moments}
Elamir and Seheult \cite{elamir03} proposed a trimmed version of univariate L-moments. Let us recall some notations. If $X_1,...,X_r$ are real-valued iid random variables, we note $X_{1:r}\leq X_{2:r}\leq \dots \leq X_{r:r}$ the order statistics. The TL-moments of order $r$ are defined for two trimming parameters $t_1$ and $t_2$ :
\begin{equation}
\lambda_r^{(t_1,t_2)} = \frac {1}{r} \sum_{k=0}^{r-1} (-1)^k \dbinom{k}{r-1} \mathbb{E}[X_{r-k+t_1:r+t_1+t_2}]
\label{def_tl}
\end{equation}
If $t_1=t_2=0$, the trimmed L-moments reduce to standard L-moments. Intuitively, we do not consider the $t_1$ first lower samples and the $t_2$ higher.

\begin{example}
Let us present some trimmed L-moments of low order :
\begin{eqnarray*}
\lambda_1^{(1,1)} &=& \mathbb{E}[X_{2:3}]\\
\lambda_1^{(1,0)} &=& \mathbb{E}[X_{1:2}]\\
\lambda_2^{(1,1)} &=& \frac{1}{2}\mathbb{E}[X_{3:4}-X_{2:3}]\\
\lambda_3^{(1,1)} &=& \frac{1}{3}\mathbb{E}[X_{4:5}-2X_{3:5}+X_{2:5}]\\
\lambda_4^{(1,1)} &=& \frac{1}{4}\mathbb{E}[X_{5:6}-3X_{4:6}+3X_{3:6}-X_{2:6}]
\end{eqnarray*}
\end{example}

The expectations of the order statistics are written in function of the quantile of the common distribution of the $X_i$'s through
\begin{equation*}
\mathbb{E}[X_{i:r}] = \frac{r!}{(i-1)!(r-i)!} \int_0^1 Q(u)u^{i-1}(1-u)^{r-i}du
\end{equation*}
We may also give the alternative definition of the TL-moments as scalar product in $L_2([0;1])$ :
\begin{equation*}
\lambda_r^{(t_1,t_2)} = \int_0^1 Q(u)P_{r}^{(t_1,t_2)}(u)du
\end{equation*}
with 
\begin{equation*}
P_{r}^{(t_1,t_2)}(u) = \frac {1}{r} \sum_{k=0}^{r-1} (-1)^k \dbinom{k}{r-1}\frac{(r+t_1+t_2)!}{(r-k+t_1-1)!(t_2+k)!}u^{r-k+t_1-1}(1-u)^{t_2+k}
\end{equation*}
It is worth noting that $\lambda_r^{(t_1,t_2)}$ may exist even if the common distribution of $X_i$ does not have a finite expectation. For example, the existence holds for Cauchy distribution of parameter $x_0\in\mathbb{R},\sigma>0$ and $t_1,t_2> 1$.\\
Also, some robustness of the sampled trimmed L-moments to outliers holds. Let $x_1,...,x_n$ an iid sample, then the empirical TL-moments are defined by the U-statistics corresponding to Definition \ref{def_tl} taking into account all subsamples of size $r$, similarly to the empirical L-moments. Hence, we remark that the $t_1$ lower and the $t_2$ larger $x_i$'s are not considered for the empirical TL-moments. Although this estimator is robust to these extreme points, the breakdown point of this sample version is $0$ because it eliminates a fixed number of extreme values, and so the proportion of eliminated samples will be asymptotically zero. It can be interesting to suppress a fixed proportion of high value in order to reinforce the robustness of the tool.

\subsubsection{An other approach for robust multivariate trimmed L-moments}
We cannot directly adapt the univariate trimmed version of L-moments to the multivariate case. Indeed, the multivariate L-moments are not expressed as linear combinations of expectations of order statistics. We then propose the following definition for trimmed L-moments : 
\begin{equation}
\lambda_{\alpha}^{(D)} = \int_D Q(t)L_{\alpha}(t)dt
\end{equation}
with $Q$ a transport of the uniform distribution in $[0;1]^d$ into the distribution of interest, $L_{\alpha}$ the multivariate Legendre polynomial of index $\alpha$ and $D$ a domain included in $[0;1]^d$. For example, the most intuitive choice would be to consider $D=[t_1,1-t_1]\times \dots\times [t_d,1-t_d]$ with $t_1,\dots,t_d\in[0;\frac{1}{2}]$ representing \textbf{the proportion of extremal deleted samples}.\\

\begin{figure}[h!]
  \centering
  \includegraphics[height=2.5in]{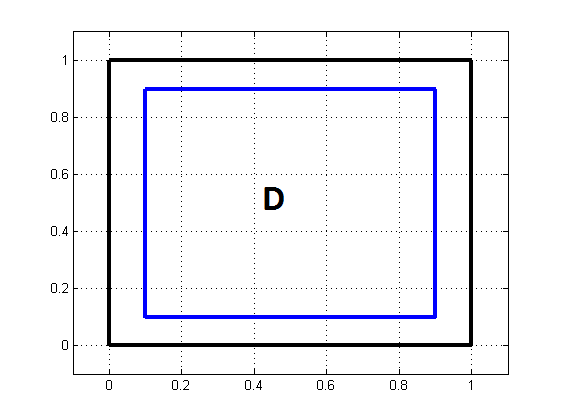}
  \caption{Area $D=[0.1;0.9]\times [0.1;0.9]\subset [0;1]^2$}
\label{trimmed_area}
\end{figure}

\begin{prop}
Let the empirical trimmed version of the L-moments issued from the optimal transport be defined by 
\begin{equation}
\hat\lambda_{\alpha}^{(D)} = \int_D Q_n(t)L_{\alpha}(t)dt
\end{equation}
where $Q_n=T_n\circ Q_0$ is defined in Theorem \ref{optimal_cons}. Then
\begin{equation}
\hat\lambda_{\alpha}^{(D)} \ps \lambda_{\alpha}^{(D)}
\end{equation}
\end{prop}
\begin{proof}
We simply perform the same inequality as in Theorem \ref{optimal_cons}
\begin{eqnarray*}
\|\hat\lambda_{\alpha}^{(D)} -\lambda_{\alpha}^{(D)}\| &\leq& \int_D \|T_n(Q_0(t))- T(Q_0(t))\|dt\\
&\leq& \int_{\Omega} \|T_n(Q_0(t))- T(Q_0(t))\|dt\\
&\leq&\|T_n-T\|_1\ps 0
\end{eqnarray*}
\end{proof}

\subsection{Hermite L-moments}

\subsubsection{Motivation}
Until now, L-moments have been defined through an inner product between a transport from the uniform measure on $[0;1]^d$ onto the measure of interest and the elements of an orthogonal basis of polynomials. This definition was motivated by the analogy with the univariate case. In the present multivariate setting, L-moments are not defined as expectations, but namely through Definition \ref{def_lmom}. This allows to consider other source distributions than the uniform one on $[0;1]^d$. It appears that a convenient choice is the standard Gaussian distribution on $\mathbb{R}^d$, which enjoys rotational invariance. In the present context, in contrast with the construction in Section \ref{section_gauss}, we do not consider the Gaussian distribution as a reference measure, but directly as the source measure.\\
As the Legendre polynomials are no longer orthogonal for this distribution, we change the basis as well. Once this basis $(P_{\alpha})$ for $\alpha\in\mathbb{N}_*^d$ is chosen, the alternate L-moments are defined by
\begin{equation*}
\lambda_{\alpha} = \int_{\mathbb{R}^d} T(x)P_{\alpha}(x)d\mathcal{N}_d(x)
\end{equation*}
where $T$ is the monotone transport from $d\mathcal{N}_d$ (the standard multivariate Gaussian measure) onto the target measure.\\
\begin{definition}
The univariate orthogonal polynomial basis $H_n$ on the space of functions measurable with respect to $d\mathcal{N}_1$, denoted by $L_2(\mathbb{R},\mathbb{R},d\mathcal{N}_1)$ and such that
\begin{equation}
\langle H_n,H_m\rangle = \int_{\mathbb{R}} H_m(x)H_n(x)d\mathcal{N}_1(x)= \sqrt{2\pi}n!\delta_{nm} \text{   for $n,m\in\mathbb{N}$}
\end{equation}
are called Hermite polynomials. A multivariate Hermite polynomial is indexed by $\alpha=(i_1,...,i_d)\in \mathbb{N}^d$:
\begin{equation}
H_{\alpha}(x) = H_{i_1}(x_1)...H_{i_d}(x_d) \text{   for $x=(x_1,...,x_d)\in\mathbb{R}^d$}
\end{equation}
\end{definition}

The first univariate Hermite polynomials are
\begin{eqnarray*}
H_0(x) &=& 1\\
H_1(x) &=& x\\
H_2(x) &=& x^2-1\\
H_3(x) &=& x^3-3x\\
H_4(x) &=& x^4-6x^2+3
\end{eqnarray*}

\begin{prop}
The family of multivariate Hermite polynomials $(H_{\alpha})_{\alpha\in\mathbb{N}^d}$ is orthogonal and complete in the Hilbert space $L^2(\mathbb{R}^d,\mathbb{R},d\mathcal{N}_d)$ provided with the scalar product :
\begin{equation}
\langle f,g\rangle = \int_{\mathbb{R}^d} f(x)g(x)d\mathcal{N}_d(x) 
\end{equation}
Moreover, for $\alpha=(i_1,...,i_d)$
\begin{equation}
\langle H_{\alpha},H_{\alpha'}\rangle = (2\pi)^{d/2}i_1!...i_d!
\end{equation}
\end{prop}
\begin{proof}
The proof is very similar to the one presented in Lemma \ref{legendre}.
\end{proof}

The Hermite L-moments are then defined by
\begin{equation}
\lambda_{\alpha} = \int_{\mathbb{R}^d} T(x)H_{\alpha}(x)d\mathcal{N}_d(x).
\label{hermite_lmom}
\end{equation}

\begin{remark}
Let us note that this definition is not compatible with the L-moments defined by Hosking in the univariate case and for $T$ the monotone transport from $d\mathcal{N}_d$ onto the target measure. Indeed, in that case, Equation \ref{hermite_lmom} is written for $\alpha=r\geq 1$
\begin{equation*}
\lambda_r = \int_0^1 Q(u) H_r(Q_{\mathcal{N}_1}(u))du
\end{equation*}
where $Q_{\mathcal{N}}$ and $Q$ respectively are the quantiles of the univariate standard Gaussian and the measure of interest.
\end{remark}

\subsubsection{Property of invariance/equivariance}
The main reason of defining such objects lies in the following property of invariance/equivariance.
\begin{prop}
Let $X$ be a random vector in $\mathbb{R}^d$ and $\nabla\varphi$ be the optimal transport from $\mathcal{N}_d$ onto the measure associated to $X$ such that $\varphi$ is convex. Let denote by $\lambda_{\alpha}(X)$ the Hermite L-moments of $X$.\\
First, let $\sigma>0, m\in\mathbb{R}^d$. Then
\begin{equation}
\lambda_{\alpha}(\sigma X+m) = \sigma\lambda_{\alpha}(X) +m\mathds{1}_{\alpha=(1...1)}.
\end{equation}
Let us note the L-moment matrix of order two
\begin{equation}
\Lambda_2(X) = \left(\int_{\mathbb{R}^d} (\nabla\varphi)_i(x)H_1(x_j)d\mathcal{N}_d(x)\right)_{i,j=1...d} = \left(\int_{\mathbb{R}^d} (\nabla\varphi)_i(x)x_jd\mathcal{N}_d(x)\right)_{i,j=1...d}.
\end{equation}
Then, if $P$ is an orthogonal matrix (i.e. $PP^T=P^TP=I_d$),
\begin{equation}
\Lambda_2(PX) = P^T\Lambda_2(X)P.
\end{equation}
\label{equivariance_hermite}
\end{prop}

\begin{remark}
This second property seems to be particular to the L-moments of degree two.
\end{remark}

\begin{proof}
The first part is similar to the proof of Proposition \ref{equivariance1}.\\
For the second part, let us define the potential $\psi:x\mapsto \varphi(Px)$. Then $\psi$ is convex and if $N_d$ denotes a standard multivariate Gaussian random vector
\begin{equation*}
\nabla\psi(N_d) = P^T\nabla\varphi(PN_d) \eqlaw P^TX
\end{equation*}
since $N_d$ is  invariant by rotation.\\
Then, $\nabla\psi$ is the monotone transport associated to $P^TX$. Thus,
\begin{eqnarray*}
\Lambda_2(PX) &=& \int_{\mathbb{R}^d} \nabla\psi(x)x^Td\mathcal{N}_d(x)\\
&=& P^T\int_{\mathbb{R}^d} \nabla\varphi(Px)x^Td\mathcal{N}_d(x)\\
&=& P^T\int_{\mathbb{R}^d} \nabla\varphi(y)(P^{-1}y)^Td\mathcal{N}_d(y)\\
&=& P^T\Lambda_2(X)P
\end{eqnarray*}
\end{proof}

\subsubsection{Applications for linear combinations of independent variables}
The previous property is well adapted for the study of linear combinations of independent variables $Z_1,...,Z_d$.\\
We suppose that each variable $Z_i$ is normalized by its second L-moment i.e. $\lambda_2(Z_i) = 1$.\\
Let us recall from Example \ref{lciv1} that if $(e_1,...,e_d)$ is an orthonormal basis of $\mathbb{R}^d$, $(b_1,...,b_d)$ the canonical basis, then the following potential
\begin{equation}
\varphi(x) = \sum_{i=1}^d \sigma_i \varphi_i(x^Te_i)
\end{equation}
is convex, with each function $\varphi_i$ convex and $a_i>0$. If we denote $P:=\sum_{i=1}^d  e_ib_i^T$ and $D:=\sum_{i=1}^d \sigma_ib_ib_i^T$, then
\begin{equation}
\nabla\varphi(x) = P^TD \myvec{\varphi'_1(x^Te_1)}{\varphi'_d(x^Te_d)}.
\end{equation}
If, for each $i$, we note $\varphi'_i:=Q_i\circ\mathcal{N}_d$ with $Q_i$ the quantile of $Z_i$
\begin{equation}
Y\eqlaw P^TD Z = P^TD\myvec{Z_1}{Z_d}.
\end{equation}
Since we have $\Lambda_2(Z)=I_d$, we deduce from Proposition \ref{equivariance_hermite} that
\begin{equation}
\Lambda_2(Y) = PDP^T.
\end{equation}
Let us remark that the covariance of $Y$ is given by
\begin{equation}
Cov(Y) = P^TD\left(\begin{array}{cccc} Cov(Z_1) & 0 & \dots & 0\\ 0& \ddots & \ddots & \vdots \\ \vdots & \ddots & \ddots & 0\\ 0 & \dots & 0& Cov(Z_d)\end{array}\right)DP.
\end{equation}
The covariance and the Hermite L-moment matrix of degree $2$ share the same rotation structure for this family of distributions. A principal component analysis can then be done on either matrix for example. Likewise, we can perform a straightforward estimation of $P$ and $D$ thanks to a method based on the L-moments.\\
$Tr(\Lambda_2)$ and $\det(\Lambda_2)$ are clearly invariant with respect to the rotation matrix $P$. Furthermore, we compute some Hermite L-moments of degree $3$ in order to identify more invariants of the distribution of $Y$; this would lead to specific plug-iin estimators for either $P$ or $D$.\\
Let $\lambda_r^{(H)}(Z_i)$ be the r-th univariate Hermite L-moment of $Z_i$ i.e.
\begin{equation*}
\lambda_r^{(H)}(Z_i) = \int_{\mathbb{R}} Q_i\circ \mathcal{N}_1(x)H_r(x)d\mathcal{N}_1(x) \text{   with $Q_i$ the quantile of $Z_i$}
\end{equation*}

\begin{lemma}
\label{lemma_H3}
If we denote by $\Lambda_3$ the matrix 
\begin{equation*}
\Lambda_3(Y) = \left(\int_{\mathbb{R}^d} (\nabla\varphi)_i(x)(x_j^2-1)d\mathcal{N}_d(x)\right)_{i,j=1...d} 
\end{equation*}
and if $Y$ is a linear combination of independent variables
\begin{equation*}
\Lambda_3(Y) = PD\left( P_{ji}^2 (\lambda_3^{(H)}(Z_i)-\lambda_1^{(H)}(Z_i))\right)_{i,j=1...d}
\end{equation*}
where $P=(P_{ij})_{i,j=1...d}$.
\end{lemma}

\begin{proof}
Let us recall that the optimal transport associated to $Y$ is 
\begin{equation*}
\nabla\varphi(x) = PD\myvec{T_1(x^Te_1)}{T_d(x^Te_d)}
\end{equation*}
with $T_i=Q_i\circ \mathcal{N}_1$. Then $\Lambda_3 = PDM$ with the (i,j)-th element of $M$ equal to
\begin{eqnarray*}
M_{ij} &=& \int_{\mathbb{R}^d} T_i(b_i^TP^Tx)x_j^2d\mathcal{N}_d(x)-\lambda_1(Z_i)\\
&=& \int_{\mathbb{R}^d} T_i(y_i)(b_j^TPy)^2d\mathcal{N}_d(y)-\lambda_1(Z_i)
\end{eqnarray*}
Now let $e_j'=P^Tb_j$. If $e_j'$ and $b_i$ are collinear, $b_i^TP^Tb_j=1$ then
\begin{equation*}
M_{ij} = \int_{\mathbb{R}} T_i(y)(b_i^TP^Tb_j y)^2d\mathcal{N}_1(y) -\lambda_1(Z_i) = \lambda_3^{(H)}(Z_i)-\lambda_1(Z_i).
\end{equation*}
Otherwise, let note $a_1 = b_i^Te_j' = b_i^TP^Tb_j$ and $a_2=\sqrt{1-a_1^2}$. If we complete $e_j'$ and $b_i$ with $d-2$ orthonormal vector, we can produce the change of variable corresponding to his new basis i.e.
\begin{eqnarray*}
M_{ij} &=& \int_{\mathbb{R}^2} T_i(y_1')(a_1y_1'+a_2y_2')^2 d\mathcal{N}_1(y_1')d\mathcal{N}_1(y_2') -\lambda_1(Z_i)\\
&=& a_1^2 \int_{\mathbb{R}} T_i(y_1')y_1^2 d\mathcal{N}_1(y_1') + a_2^2\int_{\mathbb{R}} T_i(y_1') d\mathcal{N}_1(y_1') -\lambda_1(Z_i)\\
&=&  (b_i^TP^Tb_j)^2 \lambda_3^{(H)}(Z_i) -(b_i^TP^Tb_j)^2\lambda_1^{(H)}(Z_i)
\end{eqnarray*}
which concludes the proof.
\end{proof}

Consider the case $d=2$. Then, let $L := diag(\lambda_3^{(H)}(Z_1)-\lambda_1^{(H)}(Z_2))$ and $D = diag(\sigma_1,\sigma_2)$. By Lemma \ref{lemma_H3}, it holds
\begin{equation*}
\left\{\begin{array}{lll}
\Lambda_3 &=& PDL(P.P)\\
\Lambda_2\Lambda_3 &=& PD^2L(P.P)\\
\Lambda_2^{-1}\Lambda_3 &=& PL(P.P)\\
\end{array}\right.
\end{equation*}
where $P.P$ denotes the Hadamard product between $P$ and $P$. We remark that $\Lambda_2^{-1}\Lambda_3$ is invariant with respect to $D$. Furthermore, if $\sigma_1\neq \sigma_2$, we can define $a,b$ such that
\begin{equation*}
 \left(\begin{array}{c} a\\b\end{array}\right) = - \left(\begin{array}{cc} \sigma_1 & 1\\ \sigma_2 & 1\end{array}\right)^{-1} \left(\begin{array}{c} \sigma_1^2\\\sigma_2^2\end{array}\right)
\end{equation*}
Then $\Lambda_2\Lambda_3+a\Lambda_3+b\Lambda_2^{-1}\Lambda_3 = 0$.

\subsubsection{Estimation of Hermite L-moments}
Let $x_1,...,x_n\in\mathbb{R}^d$ be n iid realizations of a random vector $X$ (with measure $\nu$). The estimation of Hermite L-moments uses the estimation of a monotone transport presented in Section \ref{section_estimation}. If $T_n$ is such a transport from $d\mathcal{N}_d$ onto $\nu_n=\sum_{i=1}^n \delta_{x_i}$, the estimation of the $\alpha$-th Hermite L-moment is
\begin{equation}
\hat\lambda_{\alpha} = \int_{\mathbb{R}^d} T_n(x)H_{\alpha}(x)d\mathcal{N}_d(x) = \sum_{i=1}^n \left(\int_{W_i(T_n)} H_{\alpha}(x)d\mathcal{N}_d(x)\right) x_i
\end{equation}
with the notations of Section \ref{section_estimation} i.e.
\begin{equation*}
W_i(T_n) = \left\{ x\in\mathbb{R}^d \text{   s.t.   } T_n(x)=x_i\right\}
\end{equation*}

\begin{theorem}
Let $\nu$ satisfy $\int \|x\|d\nu(x)<+\infty$. Then, we have for $\alpha\in\mathbb{N}_*^d$.
\begin{equation}
\hat\lambda_{\alpha} = \int_{\mathbb{R}^d} T_n(x)H_{\alpha}(x)d\mathcal{N}_d(x) \ps \lambda_{\alpha} = \int_{\mathbb{R}^d} T(x)H_{\alpha}(x)d\mathcal{N}_d(x)
\end{equation}
\label{hermite_cons}
\end{theorem}
\begin{proof}
The proof is very similar to the proof of Theorem \ref{optimal_cons}
\end{proof}

\subsubsection{Numerical applications}
We will present some numerical results for the estimation of L-moments and Hermite L-moments issued from the monotone transport. For that purpose, we simulate a linear combination of independent vectors in $\mathbb{R}^2$
\begin{equation*}
Y = P\myvecz{\sigma_1Z_1}{\sigma_2Z_2}
\end{equation*}
with 
\begin{equation*}
P = \frac{1}{\sqrt{2}}\left(\begin{array}{cc}-1 & 1 \\1& 1\end{array}\right)
\end{equation*}
and $Z_1,Z_2$ are drawn from a symmetrical Weibull distribution $\epsilon W_{\nu}$ where $\epsilon$ is a Rademacher random variable ($\epsilon=2B-1$ with $B$ is a Bernoulli a parameter $1/2$) and $W_{\nu}$ is a Weibull of shape parameter $\nu$ and scale parameter $1$. The density of $W_{\nu}$ is given by 
\begin{equation*}
f_{\theta}(x) = 8\nu (8x)^{\nu-1}e^{-(8x)^{\nu}}.
\end{equation*}

We perform $N=100$ estimations of the second L-moment matrix $\Lambda_2$, the second Hermite L-moment matrix $\Lambda_2^{(H)}$ and the covariance matrix $\Sigma$ for a sample of size $n=30$ or $100$. We present the results in Table \ref{numeric} through the following features
\begin{itemize}
\item The mean of the different estimates
\item The median of the different estimates
\item The coefficient of variation of the estimates $\hat\theta_1,...,\hat\theta_N$ (for an arbitrary parameter $\theta$)
\begin{equation*}
CV = \frac{\left(\sum_{i=1}^N \left(\theta_i-\frac{1}{N}\sum_{i=1}^N \theta_i\right)^2\right)^{1/2}}{\frac{1}{N}\sum_{i=1}^N \theta_i}
\end{equation*}

\end{itemize}

Table \ref{numeric} illustrates the fact that the L-moment estimator are more stable than classical covariance estimates but more biased for heavy-tailed distributions. The effects should be even more visible for moments of higher order. However, our sampled L-moments introduces a bias for small $n$ contrary to classical empirical covariance.

\begin{table}
\center
\begin{tabular}{|c|c||c|c|c|c|c|c|}
   \cline{3-8}
\multicolumn{2}{c||}{} & \multicolumn{3}{|c|}{$n=30$} & \multicolumn{3}{|c|}{$n=100$} \\
   \hline
Parameter & True Value & Mean & Median & CV & Mean & Median & CV \\
\hline
$\Lambda_{2,11}$ & 0.38 & 0.28 & 0.27 & 0.30 & 0.38 & 0.37 & 0.18\\ 
   \hline
$\Lambda_{2,12}$ & 0.19 & 0.14 & 0.13 & 0.65 & 0.20 & 0.20 & 0.33\\ 
   \hline
$\Lambda_{2,11}^{(H)}$ & 0.66 & 0.5 & 0.48 & 0.31 & 0.70 & 0.68 & 0.19\\ 
   \hline
$\Lambda_{2,12}^{(H)}$ & 0.33 & 0.25 & 0.24 & 0.67 & 0.38 & 0.37 & 0.34\\ 
   \hline
$\Sigma_{11}$ & 0.69 & 0.70 & 0.48 & 1.23 & 0.69 & 0.59 & 0.55\\ 
   \hline
$\Sigma_{12}$ & 0.55 & 0.55 & 0.29 & 1.62 & 0.54 & 0.47 & 0.67\\ 
   \hline
\end{tabular}
\caption{Second L-moments and covariance numerical results for $\nu=0.5$}
\label{numeric}
\end{table}

\bibliographystyle{plain} % style numéroté en anglais
\bibliography{bibliothese} % pour afficher la biblio

\appendix
\section{Proof of Theorem \ref{opt_discrete}}

We adapt the proof of Gu et al. \cite{gu13} to the case of an absolutely continuous probability measure $\mu$ defined on $\Omega\subset \mathbb{R}^d$. We do not assume the compactness of $\Omega$.\\
The proof is divided into four steps
\begin{itemize}
\item First, we show that the set $H=\left\{h\in\mathbb{R}^n\text{   s.t.  }vol(W_i(h)\cap\Omega)>0 \text{  for all i}\right\}$ is a non-void open convex set
\item Secondly, we show that $E_0(h) = \int_{\Omega} \phi_h(x)d\mu(x)$ is a $C^1$-smooth convex function on $H$ so that $\frac{\partial E_0(h)}{\partial h_i}=\int_{W_i(h)\cap\Omega}d\mu(x)$
\item In the third step, we show that $E_0(h)$ is strictly convex on $H_0^{(n)}=H\cap\{h\in\mathbb{R}^n\text{  s.t.  }\sum_{i=1}^n h_i=0\}$
\item Finally, we will prove that $\nabla\phi_h$ is a monotone transport
\end{itemize}

\subsection{Convexity of $H$}
Let denote by $H_i =\left\{h\in\mathbb{R}^n\text{   s.t.  }vol(W_i(h)\cap\Omega)>0\right\} $.
We can remark that the condition $vol(W_i(h)\cap\Omega)>0$ is the same as assuming that $W_i(h)\cap\Omega$ contains a non-empty open set in $\mathbb{R}^d$.\\
Furthermore, as $x_1,...,x_n$ are distinct, if $int(W_i(h))\neq \emptyset$, then $int(W_i(h)) = \{u\in\mathbb{R}^d\text{   s.t.  }u.x_i+h_i>\max_{j\neq i} u.x_j+h_j\}$ (Prop 2.2(a) of Gu et al. \cite{gu13}). It follows that
\begin{equation*}
H_i =\left\{h\in\mathbb{R}^n\text{   s.t.  there exists $u\in\Omega$ so that } u.x_i + h_i>\max_{j\neq i}u.x_j+h_j\right\}.
\end{equation*}
We prove now that, for any $i$, $H_i$ is convex which implies that $H=\cap_{i=1}^n H_i$ is convex.\\
Let $\alpha,\beta\in H_i$ and $0\leq t\leq 1$. Then there exists $v_1,v_2\in \Omega$ such that $v_1.x_i+\alpha_i>v_1.x_j+\alpha_j$ and $v_2.x_i+\beta_i>v_2.x_j+\beta_j$ for $j\neq i$. Then
\begin{equation*}
(tv_1+(1-t)v_2).x_i+ (t\alpha_i + (1-t)\beta_i)>(tv_1+(1-t)v_2).x_j+ (t\alpha_j + (1-t)\beta_j) \text{ for all $j\neq i$}
\end{equation*}
i.e. $tv_1+(1-t)v_2\in H_i$ i.e. $H_i$ is convex. As $H_i$ is non empty since we can take an $h\in\mathbb{R}^n$ such that $h_i$ is as large as needed, we thus have proved that $H$ is an open convex set.\\
Furthermore, as $vol(\Omega)>0$, there exists a cube included in $\Omega$. We can then translate and rescale the Voronoi cells by the method given in Proposition \ref{init_opt} in order to prove that $H\neq\emptyset$.

\subsection{Convexity of $h \mapsto E_0(h)$ and the expression of its gradient}

Let us recall that $E_0(h) = \int_{\Omega} \phi_h(u)d\mu(u)$ with 
\begin{equation*}
\phi_h(u) = \max_{1\leq i\leq n}\{ u.x_i + h_i\} \text{    for $u\in\Omega$}
\end{equation*}
Since functions $(u,h)\mapsto u.x_i+h_i$ are linear, it follows that $(u,h)\mapsto \max_{i} u.x_i+h_i$ is convex in $\Omega\times\mathbb{R}^n$. Furthermore $d\mu$ is a positive measure. We then have that $E_0$ is convex in $\mathbb{R}^n$.\\
Now let $h,d\in\mathbb{R}^n$ and $t>0$. We consider the Gateaux derivative of $E_0$
\begin{equation*}
\frac{E_0(h+td)-E_0(h)}{t} = \int_{\Omega} \frac{\phi_{h+td}(u)-\phi_h(u)}{t} d\mu(u).
\end{equation*}
Since $\phi_h$ is piecewise linear, for almost every $u\in\Omega$, there exists $i$ such that $u\in int(W_i(h))$. Let us choose such a $u$. Then, clearly, for $t$ small enough $\phi_{h+td}(u) = u.x_i + h_i+td_i$ i.e.
\begin{equation*}
\frac{\phi_{h+td}(u)-\phi_h(u)}{t} \rightarrow_{t\rightarrow 0} d_i.
\end{equation*}
Furthermore, if we take an arbitrary $t>0$, there exists $j$ such that $\phi_{h+td}(u) = u.x_j + h_j+td_j$. Then 
\begin{equation*}
\frac{|\phi_{h+td}(u)-\phi_h(u)| }{t} = \frac{|u.(x_i-x_j) + h_i-h_j-td_j| }{t} \leq  \frac{\|u\|\|x_i-x_j\| + |h_i-h_j-td_j|| }{t}.
\end{equation*}
As $d\mu$ is a probability measure of finite expectation, $u\mapsto\frac{\|u\|\|x_i-x_j\| + |h_i-h_j-td_j|| }{t}$ is $d\mu$-measurable. Thus, by the dominated convergence theorem,
\begin{equation*}
\frac{E_0(h+td)-E_0(h)}{t} \rightarrow_{t\rightarrow 0} \sum_{i=1}^n  d_i\int_{W_i(h)\cap\Omega}d\mu(u) 
\end{equation*}
i.e. $E_0$ is Gateaux differentiable and $\frac{\partial E_0}{\partial h_i}=\int_{W_i(h)\cap\Omega}d\mu(u)$.\\
This show furthermore that for some $a\in\mathbb{R}^n$, $E_0$ could be written
\begin{equation*}
E_0(h) = \int_{a}^h \sum_{i=1}^n \int_{W_i(h)\cap\Omega}d\mu(u)dh_i.
\end{equation*}

\subsection{Strict convexity on $H_0^{(n)}$}

Let $w_i(h) = \frac{\partial E_0}{\partial h_i}=\int_{W_i(h)\cap\Omega}d\mu(u)$. Then $\sum_{i=1}^n w_i(h) = \int_{\Omega}d\mu(u) = 1$. 

\begin{lemma}
$h\mapsto w_i(h)$ is a differentiable function and if $W_i(h)\cap\Omega$ and $W_j(h)\cap\Omega$ share a face F with codimension 1. Then
\begin{equation}
\frac{\partial w_i(h)}{\partial h_j}  = -\frac{1}{\|x_i-x_j\|}\int_{F} d\mu_F(u) \text{    for $j\neq i$}
\end{equation}
Otherwise 
\begin{equation}
\frac{\partial w_i(h)}{\partial h_j}  = 0 \text{    for $j\neq i$}.
\end{equation}
\label{gu2}
\end{lemma}
\begin{proof}
It is the adaptation of the proof of Gu et al. \cite{gu13} by replacing the compactness assumption by the hypothesis that $d\mu$ is a probability measure. The use of dominated convergence theorem remains unchanged.
\end{proof}
We have proved that $E_0$ is twice differentiable and we note the Hessian matrix of $E_0$
\begin{equation*}
Hess(E_0) = \left[a_{ij}\right]_{i,j=1...n}=\left[\frac{\partial^2 E_0}{\partial h_i\partial h_j}\right]_{i,j=1...n} = \left[\frac{\partial w_i(h)}{\partial h_j}\right]_{i,j=1...n}.
\end{equation*}
Then 
\begin{equation*}
\sum_{i=1}^n \frac{\partial^2 E_0}{\partial h_i\partial h_j} = \frac{\partial}{\partial h_j}1 = 0
\end{equation*}
i.e. $(1,...,1)\in Ker(Hess(E_0))$. Furthermore, this equality combined with Lemma \ref{gu2} shows that $Hess(E_0)$ is diagonally dominant with positive diagonal entries i.e.
\begin{equation*}
a_{ii} = Hess(E_0)_{ii} =  -\sum_{j\neq i}a_{ij}\geq 0.
\end{equation*}
As $Hess(E_0)$ is Hermitian, $Hess(E_0)$ is positive semidefinite. It remains to show that $(1...1)$ is the only member of its kernel.\\
Let $y\in Ker(Hess(E_0))$. Let us assume without loss of generality that $y_1=\max_{i=1...n}|y_i|>0$. Then if we combine the two equalities
\begin{equation*}
a_{11}y_1 = -\sum_{j=2}^n a_{1j}y_j
\end{equation*}
and
\begin{equation*}
a_{11} = -\sum_{j=2}^n a_{1j}
\end{equation*}
we get
\begin{eqnarray*}
\sum_{j=2}^n a_{1j}(y_j-y_1) = 0
\end{eqnarray*}
As $a_{1j}\leq 0$ and $y_j\leq y_1$, either $a_{1j}=0$ either $y_j=y_1$.\\
Since $\Omega$ is a convex domain and each $W_k(h)$ as well, there exists a rearrangement of $(1,...,n)$, denoted by $i_1,..., i_n$, such that $W_{i_j}\cap\Omega$ and $W_{i_{j+1}}\cap\Omega$ share a codimension-1 face for each $j$. We can again assume without loss of generality that $i_1=1$. Then by iteration, we find that $y_{i_{j+1}} = y_1$ since $a_{i_ji_{j+1}}<0$ for any $j$.\\
It follows that $y = y_1(1....1)$ i.e. $\dim(Ker(Hess(E_0))=0$.

\subsection{$\nabla\phi_{h^*}$ is an optimal transport map}
We produce here the proof of Aurenhammer et al. \cite{aurenhammer98} in order to prove that $\nabla\phi_{h^*}$ minimizes the quadratic transport cost.\\
Let first remark that the quadratic transport of $\nabla\phi_{h^*}$ is 
\begin{equation*}
\sum_{i=1}^n \int_{W_i(h^*)} \|x_i-u\|^2d\mu(u).
\end{equation*}
From Lemma \ref{power_lemma}, $\cup_{i=1}^n W_i(h^*)$ is the power diagram associated to $(x_1,w_1=-\|x_1\|^2-2h_1^*),...,(x_n,w_n=-\|x_n\|^2-2h_n^*)$. Suppose that $(V_1,...,V_n)$ is any partition of $\mathbb{R}^d$ such that 
\begin{equation*}
\int_{V_i} d\mu(u) = \int_{W_i(h^*)}d\mu(u) = \frac{1}{n}  \text{   for any $i=1...n$}
\end{equation*}
By definition of the power diagram, we get
\begin{equation*}
\sum_{i=1}^n \int_{W_i(h^*)} (\|x_i-u\|^2+w_i) d\mu(u) \leq \sum_{i=1}^n \int_{V_i} (\|x_i-u\|^2+w_i) d\mu(u) 
\end{equation*}
i.e.
\begin{equation*}
\sum_{i=1}^n \int_{W_i(h^*)} \|x_i-u\|^2 d\mu(u) \leq \sum_{i=1}^n \int_{V_i} \|x_i-u\|^2 d\mu(u).
\end{equation*}
This shows that $\nabla\phi_{h^*}$ minimized the quadratic cost.\\
Alternatively, we could mention that as $\nabla\phi_{h^*}$ is a gradient of a convex function which transports $\mu$ onto $\nu_n$, we get the result above by Proposition \ref{prop_brenier}. However, the proof given above makes this result explicit.

\end{document}